\documentclass[11pt]{amsart}

 \usepackage{amsfonts,graphics,amsmath,amsthm,amsfonts,amscd,amssymb,amsmath,latexsym,multicol,
 mathrsfs}
\usepackage{epsfig,url}
\usepackage{flafter}
\usepackage{fancyhdr}
\usepackage{hyperref}
\hypersetup{colorlinks=true, linkcolor=black}
\usepackage{xcolor}
\usepackage{tikz}
\usetikzlibrary{graphs,graphs.standard}

 \usepackage[matrix, arrow]{xy}

\DeclareMathOperator{\Div}{Div}

\DeclareMathOperator{\mld}{mld}

\DeclareMathOperator{\Proj}{Proj}

\DeclareMathOperator{\Spec}{Spec}
\DeclareMathOperator{\Supp}{Supp}

 \numberwithin{equation}{subsection}
 \numberwithin{footnote}{subsection}

 \newtheorem{cor}[subsection]{Corollary}
 \newtheorem{lem}[subsection]{Lemma}
 \newtheorem{prop}[subsection]{Proposition}
 \newtheorem{thm}[subsection]{Theorem}

 \newtheorem{quest}[subsection]{Question}

{
\theoremstyle{upright}
\theoremstyle{definition}
 \newtheorem{defn}[subsection]{Definition}
 \newtheorem{exa}[subsection]{Example}
 
 \newtheorem{rem}[subsection]{Remark}

}

 \newcommand{\PP}{\mathbb P}
 \newcommand{\A}{\mathbb A}
 \newcommand{\Q}{\mathbb Q}
 \newcommand{\R}{\mathbb R}
 \newcommand{\Z}{\mathbb Z}
  \newcommand{\C}{\mathbb C}
 \newcommand{\bir}{\dashrightarrow}
\title{\large S\MakeLowercase{ingularities of rational maps: foundations and surfaces}}
  %— Toric varieties, Newton polyhedra, Okounkov bodies}

\author{\large C\MakeLowercase{aucher} B\MakeLowercase{irkar}}
\date{\today}
\begin{document}
\maketitle

\begin{abstract}
We develop a theory of singularities of rational maps, focusing on maps
$$
f\colon X\dashrightarrow \mathbb P^n,
$$
and using their polarised graphs and normalised polarised graphs even when the source $X$ is very singular. 
To measure singularities of the map at a point $x\in X$, i.e. how far it is from being regular, we introduce invariants including the normalised graph fibre degree $\delta_x(f)$, a generalised lc threshold $\lambda_x(f)$, and invariants measuring the singularities of the normalised graph itself. Numerous examples show that the resulting invariants measure genuinely different aspects of map singularities.

We investigate the surface case in detail. We relate the normalised graph fibre degree to multiplicity, prove a sharp threshold--degree inequality for klt surface germs, and develop a detailed theory of linear type maps on smooth and singular surfaces. In particular, we connect the existence of linear type maps to existence of smooth curves through the given point, and with local class groups and complement theory. 

We conclude with questions and future directions concerning higher dimensions, complements and boundedness, moduli, Cremona groups, commutative algebra, curve-counting theories, and positive characteristic.
\end{abstract}

\tableofcontents

%%%%%%%%%%%%%%%%%%%%%%%%
%%%%%%%%%%%%%%%%%%%%%%%%%%

\section{\bf Introduction}

 Unless stated otherwise, we work over an algebraically closed field $k$
of characteristic zero.

The purpose of this paper is to initiate a theory of singularities of rational
maps. Let
$$
f\colon X\bir Z
$$
be a rational map between algebraic varieties. We say that $f$ is
\emph{regular} at a point $x\in X$ if it is defined at $x$, and
\emph{singular} at $x$ otherwise. Thus the singular locus of $f$ is its
indeterminacy locus. Our aim is not only to determine where $f$ is singular,
but to measure how singular it is.

One aim of this paper is to illustrate that this leads to a genuinely
non-trivial and potentially deep theory. Already for maps from surfaces to the projective space, one
finds natural numerical and geometric invariants, thresholds, boundedness statements, classification of
the mildest singularities, and connections with multiplicity, local class
groups and complement theory.

A guiding principle throughout is that the theory should not be restricted
to maps from smooth varieties, or even from varieties with mild singularities. The singularities
of the source and the singularities of the map are different kinds of data.
A rational map from a smooth variety may be highly singular, while a map from
a badly singular variety may have very mild singularities. Accordingly, many
of the basic constructions and several of the main results below are
formulated for normal varieties or arbitrary normal surface germs. Stronger
hypotheses such as klt, rational, or smooth are imposed only where the
relevant birational tools require them.

{\textbf{\sffamily{Main objects.}}}
We mainly consider rational maps
$$
f\colon X\bir\mathbb P^n
$$
as this setting already exhibits a rich geometry.
Let
$$
Y=\Gamma_f\subset X\times\mathbb P^n
$$
be the graph of $f$, with its reduced structure, and let
$$
\pi\colon Y\to X,\qquad
g\colon Y\to\mathbb P^n
$$
be the two projections. The second projection determines the tautological
line bundle
$$
\mathcal O_Y(1):=g^*\mathcal O_{\mathbb P^n}(1).
$$
We call
$$
\xymatrix{
& Y,\mathcal O_Y(1) \ar[dl]_{\pi} \ar[dr]^{g} & \\
X && \mathbb P^n
}
$$
the \emph{polarised graph} of $f$. The polarisation is essential: the
birational morphism $Y\to X$ alone does not in general determine the
singularity of the rational map.

The graph is often non-normal, even when $X$ is smooth. Our main approach is
therefore to pass to its normalisation
$$
\nu\colon Y^\nu\to Y.
$$
Put
$$
p:=\pi\circ\nu,\qquad q:=g\circ\nu,
$$
and
$$
\mathcal O_{Y^\nu}(1)
:=
\nu^*\mathcal O_Y(1)
=
q^*\mathcal O_{\mathbb P^n}(1).
$$
This gives the \emph{normalised polarised graph}
$$
\xymatrix{
& Y^\nu,\mathcal O_{Y^\nu}(1) \ar[dl]_{p} \ar[dr]^{q} & \\
X && \mathbb P^n .
}
$$

Fix a closed point $x\in X$, and denote the scheme-theoretic fibre of $p$
over $x$ by
$$
Y^\nu_x.
$$
The basic local object attached to the singularity of $f$ at $x$ is the
polarised normalised graph fibre
$$
(Y^\nu_x,
\mathcal O_{Y^\nu}(1)|_{Y^\nu_x})\to \PP^n
$$
together with its natural morphism to $\mathbb P^n$. Thus one may study the
singularity of $f$ at $x$ at several levels: through the polarised graph, the
normalised polarised graph, or the polarised normalised graph fibre.

This point of view is closely related to the classical theory of ideals and
blowups. If $X$ is affine and $f=(f_0:\cdots:f_n)$ is represented by regular functions, 
 then the graph is the blowup of the corresponding ideal
$$
I=(f_0,\ldots,f_n).
$$
Thus the graph is governed by the Rees algebra of $I$. The advantage of the
polarised graph is that it is intrinsic to the rational map and remains
available at singular points where there may be no preferred base ideal.

{\textbf{\sffamily{Main invariants.}}}
We study three main kinds of invariants. The first comes directly from the
geometry of the polarised fibre. Its simplest numerical invariant is the
\emph{normalised graph fibre degree}
$$
\delta_x(f)
:=
\deg_{\mathcal O_{Y^\nu}(1)}Y^\nu_x.
$$
More refined information is contained in the Hilbert polynomial of
$$
\left(
Y^\nu_x,
\mathcal O_{Y^\nu}(1)|_{Y^\nu_x}
\right)
$$
and in the degree spectrum obtained from cycles supported on $Y^\nu_x$.

The second kind of invariant is given by thresholds. Assume that $X$ is klt
and $f^*H$ is $\mathbb Q$-Cartier near $x$, where
$H\subset\mathbb P^n$ is a hyperplane. Viewing $(X,tf^*H)$ as a
generalised pair whose nef part is induced by $tq^*H$, we define
$$
\lambda_x(f)
:=
\sup\{t\geq0\mid
(X,tf^*H)\text{ is generalised lc at }x\}.
$$
This is independent of the choice of the hyperplane $H$. At
a smooth point it agrees with the log canonical threshold of the primitive
local base ideal of $f$. We will also use generalised canonical thresholds, which are more closely related to classical birational geometry, in particular the Noether--Fano method.

A third measure comes from the singularities of the normalised graph itself.
When $K_{Y^\nu}$ is $\mathbb Q$-Cartier, we consider
$$
\theta_x(f)
:=
\inf\{a(E,Y^\nu,0)\mid
E\text{ is a prime divisor over }Y^\nu
\text{ whose centre on }X\text{ contains }x\}.
$$
Thus $\theta_x(f)\geq0$ precisely when $Y^\nu$ is lc over a neighbourhood of
$x$. For surfaces we also use numerical log discrepancies, so this approach
does not require us to assume in advance that $K_{Y^\nu}$ is
$\mathbb Q$-Cartier.

These invariants measure different aspects of the rational map.
The degree $\delta_x(f)$ measures the polarised fibre, $\lambda_x(f)$ measures
the interaction between the map and the singularities of $X$, while
$\theta_x(f)$ measures the intrinsic singularities created on the normalised
graph.

{\textbf{\sffamily{Main results.}}}
We now summarise some of the main results without giving full statements.

\medskip
\noindent
\emph{Fibre degree, multiplicity and thresholds.}
Let $X$ be a normal surface and let
$$
f\colon X\bir\mathbb P^n
$$
be singular at $x$. For a general hyperplane $H\subset\mathbb P^n$, we prove
$$
\delta_x(f)=\mu_xf^*H;
$$
see Theorem \ref{t-norm-graph-fibre-degree-mult-hyperplane-pullback}.
Thus the normalised graph fibre degree is exactly the Hilbert--Samuel
multiplicity of a general hyperplane pullback. When $f$ is represented by a
base ideal, this is closely related to its first polar multiplicity.

If $x\in X$ is klt, we prove the sharp inequality
$$
\delta_x(f)\lambda_x(f)\leq2;
$$
see Theorem \ref{t-lct-degree-inequality-for-rational-maps-on-surfaces}.
Hence a positive lower bound for the generalised lc threshold gives an upper
bound for the normalised graph fibre degree. The converse fails: there are
families with
$$
\delta_x(f)=2
$$
but $\lambda_x(f)$ tending to zero (Example \ref{exa-bnd-fib-degree-small-threshold}). The thresholds $\lambda_x(f)$ also
inherit the ACC from the ACC for generalised lc thresholds.

When $X$ is a smooth surface, there is a stronger equality. If
$\lambda_x^1(f)$ denotes the generalised $1$-lc threshold, then by Lemma \ref{l-rational-maps-mult=fib-degree=inverse-can-threshold},
$$
\mu_xf^*H
=
\delta_x(f)
=
\frac{1}{\lambda_x^1(f)}.
$$
Moreover, the degree spectrum records the corresponding thresholds appearing
after successive point blowups; see Theorem \ref{t-degree-spectrum-infinitely-near-thresholds}.

\medskip
\noindent
\emph{Singularities of the normalised graph.}
The numerical invariants above need not control the singularities of
$Y^\nu$. We construct maps from $\mathbb A^2$ for which
$$
\delta_o(f)=2,\qquad \lambda_o(f)=1,
$$
while $\theta_o(f)$ tends to zero where $o$ is the origin (Examples \ref{exa-map-degree-two-on-A2-Y-deep-klt} and \ref{exa-toric-map-degree-two-A2-Y-deep-klt}). We also give examples with smooth source
whose normalised graph is not lc (Example \ref{exa-map-degree-5-A2-non-lc-Y}).

In the opposite direction, the fibre degree imposes some boundedness on the
exceptional geometry. If $X$ is a smooth surface, $f$ is singular only at
$x$, and
$$
\delta_x(f)\leq d,
$$
then $Y^\nu$ is $(3d+1)$-bounded in the sense that, on its minimal
resolution, the self-intersections and the number of exceptional curves of
self-intersection at most $-3$ are uniformly controlled (Theorem \ref{t-bnd-singularity-negativity-normalised-graphs}).

Further examples show that a map can have the
mildest possible map singularity even when neither $X$ nor $Y^\nu$ is close
to being smooth (Example \ref{exa-toric-map-degree-one-X-Y-deep-sing}).

\medskip
\noindent
\emph{Linear type maps.}
The simplest map singularities are those with $\delta_x(f)=1$. 
We consider a slightly stronger type of map singularity, that is, those of
\emph{linear type}. We say that $f$ is linear type at $x$ if
$$
\left(
Y^\nu_x,
\mathcal O_{Y^\nu}(1)|_{Y^\nu_x}
\right)
\simeq
\left(
\mathbb P^l,
\mathcal O_{\mathbb P^l}(1)
\right)
$$
for some $l\geq0$. For a singular map from a surface, necessarily $l=1$.

For a map $f\colon X\bir \PP^n$ from a smooth surface, linear type admits a strikingly rigid
characterisation (Theorem \ref{t-linear-type-maps-smooth-surfaces}). In particular, assuming $f$ is singular at $x$
the following are equivalent:
\begin{itemize}
\item $f$ is linear type at $x$;
\item the graph $Y$ is normal with canonical singularities, $Y_x\simeq\PP^1$,
$\mathcal{O}_Y(1)|_{Y_x}\simeq\mathcal{O}_{\PP^1}(1)$;
\item the graph $Y$ is normal with canonical singularities, $-K_Y$ is ample over $X$, and
$
\deg\mathcal{O}_Y(1)|_{Y_x}=1;
$
\item $\lambda_x^1(f)=1$;
\item $f^*H$ is smooth at $x$ for a general hyperplane $H\subset\mathbb P^n$.
\end{itemize}

The theory becomes subtler when the source $x\in X$ is singular. If $x\in X$ is a
rational surface singularity and $f$ is singular at $x$, then
$$
f\text{ is linear type at }x
\quad\Longleftrightarrow\quad
\delta_x(f)=1;
$$
see Theorem \ref{t-degree-one-linear-type-rational-surfaces}.

For an arbitrary normal surface germ $x\in X$, with no rationality or lc assumption, we
prove that 
there exists $f\colon X\bir \PP^n$ singular at $x$ with $\delta_x(f)=1$
if and only if there exists a curve $D\subset X$ through $x$ which is smooth at
$x$ (Theorem \ref{t-degree-one-maps-smooth-curves}). Moreover, one can take the target to be $\mathbb P^1$.

On the other hand, every affine toric surface germ at its torus-fixed point admits a toric
singular linear type map to $\mathbb P^1$ (Theorem \ref{t-linear-type-maps-on-toric-surfaces} and Corollary \ref{c-toric-surfaces-admit-linear-type-map}). 

Linear type maps need not exist even on canonical surface singularities. Over
$k=\mathbb C$, for a singular canonical surface germ we prove
$$
x\in X\text{ admits a singular linear type map}
\quad\Longleftrightarrow\quad
\operatorname{Cl}(\mathcal O_{X,x})\neq0;
$$
see Corollary \ref{c-canonical-linear-type-local-class-group}.
Thus the existence of a linear type map can depend on the algebraic local
class group and need not be determined by the completed local singularity. 

We also further investigate linear type maps on $A$-type and $D$-type klt singularities and discuss connection with complement theory (see \ref{ss-A-type-klt-sing-lin-type} and \ref{ss-D-type-klt-sing-lin-type}).

{\textbf{\sffamily{Further developments.}}}
Taken together, these results illustrate that singularities
of rational maps support a substantial birational theory rather than merely
a reformulation of the theory of base loci. The theory has numerical,
valuation-theoretic and geometric invariants; mild classes admit classification; boundedness phenomena occur; and interactions with classical singularities,
ideals, class groups and complements appear. 

Many basic problems remain open, particularly in higher dimension. In the
final section we discuss boundedness and anti-canonical maps, restriction and
adjunction, composition and moduli, relations with ordinary singularities and
base ideals, and possible connections with Cremona groups, enumerative geometry, and other areas.

{\textbf{\sffamily{Structure of the paper.}}}
We finish by describing the organisation of the paper. Section 2 recalls
background on singularities of pairs, generalised pairs and rational
singularities. Section 3 develops polarised graphs, normalised graphs, base
ideals, Rees algebras and pullbacks of divisors. Section 4 discusses maps on
curves and multiplicity. Section 5 introduces the normalised graph fibre
degree and degree spectrum, proves the comparison with multiplicity, and
gives toric formulas and examples. Section 6 introduces the generalised lc
threshold and discusses ACC and explicit computations. Section 7 studies
the singularities of normalised graphs and presents several examples. Section 8 proves the generalised lc
threshold--degree inequality for klt surfaces. Section 9 develops the theory
for maps from smooth surfaces, including characterisation of linear type maps, and bounded negativity. Section 10 studies linear type maps
on singular surfaces treating toric, canonical and klt cases, and explores connections with other problems. Section 11 collects questions, future directions and
possible applications.

\section*{Acknowledgements}

This work was initiated in 2025 and developed over an extended period with support from a grant of Tsinghua University and a grant of the National Program of Overseas High Level Talent.
The main ideas, definitions, questions, and overall direction of this work were developed by the author. AI tools (ChatGPT and DeepSeek) were used during the preparation of the paper as technical assistants to help with checking calculations and arguments, developing details of proofs and examples using standard techniques, locating relevant references, and improving the exposition. All mathematical statements, proofs, and conclusions were verified by the author.

%%%%%%%%%%%%%%%%%%%%%%%%%%%%%%
%%%%%%%%%%%%%%%%%%%%%%%%%%%%%%
\section{\bf Preliminaries}

We work over an algebraically closed field $k$ of characteristic zero unless stated otherwise. 

\subsection{Singularities of pairs}

We recall some standard terminology from birational geometry; see for example
\cite{Kollar-Mori}. A pair $(X,B)$ consists of a normal variety $X$ and an
$\mathbb R$-divisor $B\geq 0$ such that $K_X+B$ is $\mathbb R$-Cartier.
Let
$$
\phi\colon W\to X
$$
be a log resolution and write
$$
K_W+B_W=\phi^*(K_X+B).
$$
If $E$ is a prime divisor on $W$, its log discrepancy with respect to $(X,B)$ is
$$
a(E,X,B):=1-\mu_E B_W.
$$
More generally, this defines $a(E,X,B)$ for every prime divisor $E$ over $X$,
independently of the choice of resolution. We denote by $c_X(E)$ the centre of
$E$ on $X$.

We say that $(X,B)$ is $\epsilon$-log canonical ($\epsilon$-lc) if
$$
a(E,X,B)\geq \epsilon
$$
for every prime divisor $E$ over $X$. We say that $(X,B)$ is log canonical (lc)
if it is $0$-lc, Kawamata log terminal (klt) if it is $\epsilon$-lc for some
$\epsilon>0$, and canonical if it is $1$-lc. When $B=0$, we simply say that
$X$ is lc, klt, or canonical.

For a closed point $x\in X$, the minimal log discrepancy is
$$
\operatorname{mld}_x(X,B)
:=
\inf\{a(E,X,B)\mid C_X(E)=x\}.
$$

If $D\geq 0$ is an $\mathbb R$-Cartier divisor, its log canonical threshold at
$x$ with respect to $(X,B)$ is
$$
\operatorname{lct}_x(X,B,D)
:=
\sup\{t\geq 0\mid (X,B+tD)\text{ is lc near }x\}.
$$

\subsection{Generalised pairs}
We also use generalised pairs; for an introduction see \cite{Birkar-generalised-pairs}.
A \emph{generalised pair} $(X,B+M)$ consists of 
\begin{itemize}
\item a projective morphism $X\to Z$ of normal varieties, 

\item an $\R$-divisor $B\ge 0$ on $X$, and 

\item a birational contraction $\phi\colon X'\to X$ and an $\R$-divisor $M'$ on $X'$ which is nef over $Z$ 
\end{itemize}
such that $K_{X}+B+M$ is $\R$-Cartier where $M:= \phi_*M'$. 

Actually we specify $X',M'$ only up to birational transformations, that is, 
if we replace $X'$ with a resolution and replace $M'$ with its pullback, then the pair would be the same.
In other words, we view $M'$ as a so-called b-divisor which is determined by its 
trace on the model $X'$. When $X\to Z$ is the identity (which is the case mostly in this paper), we usually drop $Z$.

Now we define generalised singularities for a generalised pair $(X,B+M)$.
Replacing $X'$ we can assume $\phi$ is a log resolution of $(X,B)$. We can write 
$$
K_{X'}+B'+M'=\phi^*(K_{X}+B+M)
$$
for some uniquely determined $B'$. 
For a prime divisor $E$ on $X'$, the generalised log discrepancy is
$$
a(E,X,B+M):=1-\mu_E B'.
$$

The notions of generalised $\epsilon$-lc, lc, klt and canonical are defined in
the same way as above using these generalised log discrepancies.

If $D$ is an effective $\R$-divisor on $X$ and $N'$ is an $\R$-divisor that is nef over $Z$ so that $D+N$ is $\R$-Cartier where $N=\phi_*N'$, the generalised lc threshold of $D+N$ with respect to $(X,B+M)$ is 
$$
\sup\{t\geq 0\mid (X,B+M+tD+tN)\text{ is generalised lc}\}
$$
where the nef part of the pair in the deifinition is $M'+tN'$.

\subsection{Rational singularities}

A normal variety $X$ has rational singularities if for a resolution
$$
\phi\colon W\to X
$$
we have
$$
R^i\phi_*\mathcal{O}_W=0
$$
for every $i>0$. This condition is independent of the
choice of resolution (as we work in characteristic zero). We will use the standard fact that klt singularities are
rational; see \cite{Kollar-Mori}.

\subsection{Finite morphism fibrewise closed immersion}

\begin{lem}\label{lem-finite-morphism-fibrewise-closed-immersion}
Let
$$
\xymatrix{
Y \ar[r]^{\pi} \ar[dr] & Z \ar[d] \\
& T
}
$$
be a finite morphism of schemes over $T$. Assume that for every point
$t\in T$, the induced morphism on fibres
$$
\pi_t\colon Y_t\to Z_t
$$
is a closed immersion. Then $\pi\colon Y\to Z$ is a closed immersion.
\end{lem}

\begin{proof}
Since $\pi$ is finite, it is affine. Therefore $\pi$ is a closed immersion
if and only if the natural morphism
$$
\mathcal{O}_Z\longrightarrow \pi_*\mathcal{O}_Y
$$
is surjective.

Let
$$
\mathcal{C}
=
\operatorname{coker}
\left(
\mathcal{O}_Z\longrightarrow \pi_*\mathcal{O}_Y
\right).
$$
This is a coherent sheaf on $Z$. We claim that $\mathcal{C}=0$.

Since $\pi$ is finite, formation of $\pi_*\mathcal{O}_Y$ commutes with
arbitrary base change. Hence, for every $t\in T$, the restriction of
$$
\mathcal{O}_Z\longrightarrow \pi_*\mathcal{O}_Y
$$
to the fibre $Z_t$ is the natural morphism
$$
\mathcal{O}_{Z_t}
\longrightarrow
(\pi_t)_*\mathcal{O}_{Y_t}.
$$
By assumption, $\pi_t$ is a closed immersion, so this map is surjective.
Therefore
$$
\mathcal{C}|_{Z_t}=0
$$
for every $t\in T$.

Let $z\in Z$, and let $t\in T$ be its image. Then
$$
\mathcal{C}_z\otimes_{\mathcal{O}_{T,t}} k(t)=0.
$$
Equivalently,
$$
\mathcal{C}_z/{m}_t\mathcal{C}_z=0.
$$
Since
$$
{m}_t\mathcal{O}_{Z,z}\subseteq {m}_z,
$$
Nakayama's lemma implies
$$
\mathcal{C}_z=0.
$$
Thus $\mathcal{C}=0$, and so
$$
\mathcal{O}_Z\longrightarrow \pi_*\mathcal{O}_Y
$$
is surjective. Hence $\pi$ is a closed immersion.
\end{proof}

%%%%%%%%%%%%%%%%%%%%%%%%%%%%%%
%%%%%%%%%%%%%%%%%%%%%%%%%%%%%%
\section{\bf Graphs of rational maps and related topics}

In this section we collect background on rational maps, their graphs, linear systems, base ideals, Rees algebras and blowups, together with some basic examples used later. Much of this material is well known, although we occasionally work in settings less commonly discussed, for instance with singular source varieties where base ideals and divisors require additional care; see \cite{Hartshorne} and \cite{Lazarsfeld-Positivity-I} for general background. 

We introduce the polarised graph and normalised polarised graph as our main geometric objects for studying singularities of rational maps. In particular, we define maps of linear type, which will be a main focus in later sections.

%%%%%%%%%%%%%%%%%%%%%%%%%%%%%%
\subsection{The polarised graph}
Let
$$
f\colon X\bir \PP^n
$$
be a rational map from a reduced scheme, and let $U$ be the maximal open subset where $f$ is regular: we are assuming that $U$ is dense in $X$. Denote by
$$
\Gamma_U\subset U\times \PP^n
$$
the graph of the morphism $f|_U$. Then $\Gamma_U$ is a closed subset and $\Gamma_U\to U$ is an isomorphism. We define
$$
Y=\Gamma_f\subset X\times \PP^n
$$
to be the closure of $\Gamma_U$ with reduced structure. Let
$$
\pi\colon Y\to X,
\qquad
g\colon Y\to \PP^n
$$
be the induced projections. Then $\pi$ is projective and birational, and an isomorphism over $U$.
Moreover, $f$ is regular at a point $x\in X$ if and only if $\pi$ is an
isomorphism over some neighbourhood of $x$.

The morphism $g$ induces the line bundle
$$
\mathcal O_Y(1):=g^*\mathcal O_{\PP^n}(1).
$$
We package the data into a diagram
$$
\xymatrix{
& Y,\mathcal O_Y(1) \ar[dl]_{\pi} \ar[dr]^{g} & \\
X && \PP^n
}
$$
and call it the \emph{polarised graph} of $f$. 
The basic idea is that the complexity of this diagram and related diagrams below reflect the defect of $f$ to be a morphism. 

%%%%%%%%%%%%%%%%%%%%%%%%%%%%%%
\subsection{The normalised polarised graph}
The graph $Y$ is often non-normal, even when $X$ is smooth. It is helpful to pass to the normalisation
$$
\nu\colon Y^\nu\to Y.
$$
In our setting $\nu$ is a finite morphism as $Y$ is reduced and of finite type over the ground field $k$. Set
$$
p:=\pi\circ \nu\colon Y^\nu\to X,
\qquad
q:=g\circ \nu\colon Y^\nu\to \PP^n,
$$
and define
$$
\mathcal O_{Y^\nu}(1):=\nu^*\mathcal O_Y(1)=q^*\mathcal O_{\PP^n}(1).
$$
Then $p$ is projective and birational, and $q$ is a morphism. We package
this data into
$$
\xymatrix{
& Y^\nu,\mathcal O_{Y^\nu}(1) \ar[dl]_{p} \ar[dr]^{q} & \\
X && \PP^n
}
$$
and call it the \emph{normalised polarised graph} of $f$.

With $U$ as above, the restriction $p^{-1}U\to U$ is the
normalisation of $U$. In particular, if $X$ is normal, then $p$ is an
isomorphism over $U$; hence in the normal case the normalised graph differs
from the graph only over the indeterminacy locus of $f$. For non-normal
$X$, the normalised graph also contains the normalisation of the locus where
$f$ is already regular.

Note that from the two morphisms $p,q$ we can recover the polarised graph, so strictly speaking no information is lost by passing to $Y^\nu$. The advantage is that we can apply more geometric tools to the normalisation.

Unless stated otherwise, below we take $X$ to be a variety, hence reduced and irreducible.

%%%%%%%%%%%%%%%%%%%%%%%%%
\subsection{Linear type maps}
\label{ss-linear-type-map}
Let $f\colon X\bir\PP^n$ be a rational map, let $x\in X$ be a closed point, and let
$$
p\colon Y^\nu\to X
$$
be the normalised graph. We say that $f$ is \emph{linear type at $x$} if, for some $l\geq0$,
$$
\left(Y^\nu_x,\mathcal{O}_{Y^\nu}(1)|_{Y^\nu_x}\right)
\simeq
\left(\PP^l,\mathcal{O}_{\PP^l}(1)\right).
$$
Equivalently, the induced morphism
$$
Y^\nu_x\to\PP^n
$$
is a linear embedding.

For maps from surfaces, $l=0$ or $1$. The case $l=0$ is equivalent to $f$ being regular at $x$, while a singular linear type map has
$$
\left(Y^\nu_x,\mathcal{O}_{Y^\nu}(1)|_{Y^\nu_x}\right)
\simeq
\left(\PP^1,\mathcal{O}_{\PP^1}(1)\right).
$$

%%%%%%%%%%%%%%%%%%%%%%%%%%%%%%
\subsection{Base ideals}
Assume $X=\Spec A$ is an affine variety, and let $x\in X$ be a closed point.
We work locally near $x$. Let
$$
f\colon X\bir \PP^n
$$
be a rational map, written as
$$
f=(f_0:\cdots:f_n)
$$
with $f_i\in K(X)$. After shrinking $X$ around $x$ and multiplying all $f_i$
by a common denominator, we may assume that $f_i\in A$. These functions define
an ideal
$$
I=(f_0,\dots,f_n)\subset A.
$$
Then the following are equivalent: 
\begin{itemize}
\item the map $f$ is regular at $x$; 
\item the graph morphism $Y\to X$ is an isomorphism over a neighbourhood of $x$; 
\item the ideal $I$ is locally free near $x$, i.e. $I$ is principal in $\mathcal O_{X,x}$. 
\end{itemize}
We first observe that local freeness of $I$ is independent of the chosen
presentation. Suppose
$$
(f_0:\dots:f_n)
\qquad\text{and}\qquad
(g_0:\dots:g_n)
$$
are two presentations of the same rational map, with all $f_i,g_i$ regular on
$X$. Since $X$ is irreducible, there is a rational function $h\in K(X)$ such
that
$$
f_i=hg_i
$$
for every $i$. Write $h=c/e$ with $c,e\in A$. Then
$$
ef_i=cg_i
$$
for every $i$, and hence
$$
e(f_0,\dots,f_n)=c(g_0,\dots,g_n).
$$
After shrinking $X$ near $x$, we may assume $c$ and $e$ are non-zero. Since
$A$ is a domain, multiplication by a non-zero element identifies an ideal with
its multiple as an $A$-module. Therefore the two ideals
$$
(f_0,\dots,f_n),\qquad (g_0,\dots,g_n)
$$
are isomorphic as $A$-modules, after shrinking $X$ near $x$ if necessary. In
particular, local freeness is independent of the chosen presentation.

We now prove the equivalence with regularity. If $f$ is regular at $x$, then
near $x$ it admits a presentation
$$
f=(g_0:\dots:g_n)
$$
where $g_i\in \mathcal O_{X,x}$ generate the unit ideal. By the observation
above, the ideal $(f_0,\dots,f_n)$ is locally free at $x$ as
$g_0,\dots,g_n$ generate $\mathcal O_{X,x}$. Conversely, if
$I=(f_0,\dots,f_n)$ is locally free at $x$, then, since we are working locally,
we may write
$$
I=(h).
$$
Thus $f_i=hg_i$ for some $g_i\in \mathcal O_{X,x}$, and the $g_i$ generate the
unit ideal. Therefore
$$
(g_0:\dots:g_n)
$$
defines a morphism near $x$, and it agrees with $f$ as a rational map.
Hence $f$ is regular at $x$.

Thus the failure of $f$ to be regular at $x$ is measured by the failure of
the ideal $I$ to be locally free at $x$.

\begin{rem}[Smooth varieties]
Now assume $x$ is a smooth point of $X$, hence $\mathcal O_{X,x}$ is a UFD. In this case
we may divide the $f_i$ by their greatest common divisor. After doing so, the
presentation of $f$ at $x$ is unique up to multiplication by a unit in $\mathcal O_{X,x}$.
Thus the corresponding local base ideal is uniquely determined by the rational
map. Conversely, let $I\subset A$ be a non-zero ideal. Choosing generators
$f_0,\dots,f_n$ of $I$ defines a rational map
$$
X\bir \PP^n,
\qquad
x\mapsto (f_0(x):\cdots:f_n(x)).
$$
\end{rem}

\begin{rem}[Normal varieties]
A standard way to define a rational map from a normal variety to projective
space is to choose sections $s_0,\dots,s_n$ of a Weil divisor $L$. Equivalently,
one views these sections as rational functions $f_i\in K(X)$ satisfying
$$
\Div(f_i)+L\geq 0.
$$
If $x\in X$ is smooth, then $L$ is Cartier near $x$, and a local trivialisation
of $L$ identifies the sections $s_i$ with regular functions. These functions
define the base ideal discussed above, uniquely up to multiplication by a unit.
Thus, at a smooth point, the singularities of the rational map, the
singularities of the base ideal, and the base locus of the corresponding
linear system are the same object from different points of view.

At a singular point this identification is more delicate. A Weil divisor need
not be Cartier, so there is no canonical local trivialisation of $L$ and hence
no canonical base ideal defined by the sections. This is one reason for working
instead with the graph, or with the normalised polarised graph.
\end{rem}

%%%%%%%%%%%%%%%%%%%%%%%%%%%%%%
\subsection{Rees algebras and blowups}
Assume $X=\Spec A$ is an affine variety,
$$
f\colon X\bir \PP^n,\qquad f=(f_0:\cdots:f_n)
$$
is a rational map with $f_i\in A$, and consider the ideal
$$
I=(f_0,\dots,f_n)\subset A.
$$
The blowup of $X$ along $I$ is
$$
B_I X:=\Proj \bigoplus_{d\geq 0} I^d.
$$
The chosen generators give a surjection of graded $A$-algebras
$$
A[t_0,\dots,t_n]\longrightarrow \bigoplus_{d\geq 0} I^d
$$
sending $t_i$ to $f_i$, where $f_i$ is viewed as an element of $I$ in degree
one. Equivalently, one may regard the Rees algebra as
$\bigoplus_{d\geq 0} I^dT^d\subset A[T]$ and send $t_i$ to $f_iT$.
Hence we get a closed embedding
$$
B_I X=\Proj \bigoplus_{d\geq 0} I^d
\hookrightarrow
\Proj A[t_0,\dots,t_n]=X\times \PP^n.
$$
Under this embedding, $\mathcal O_{\PP^n_X}(1)$ restricts to the tautological
line bundle $\mathcal O_{B_I X}(1)$ of the blowup.

If $f$ is regular at a point $y$, then $I$ is locally free at $y$, hence
$B_I X\to X$ is an isomorphism over a neighbourhood of $y$. Therefore
$$
B_I X\to X
$$
is an isomorphism over the regular locus of $f$. Under the above embedding into
$X\times \PP^n$, the blowup agrees with the graph over this locus, and hence
it is the closure of the graph. Thus, fixing a closed point $x$ and shrinking $X$ around it if necessary, the diagram
$$
\xymatrix{
& B_I X,\mathcal O_{B_I X}(1) \ar[dl] \ar[dr] & \\
X && \PP^n
}
$$
coincides with the polarised graph of $f$.

%%%%%%%%%%%%%%%%%%%%%%%%%%%%%%%%%%%
\subsection{Free resolutions and blowups}
Let $X=\Spec A$, where $A$ is a regular domain of dimension two, and let
$I=(f,g)\subset A$. Assume that $f$ and $g$ have no common irreducible
component. Then we have the free resolution
$$
0\to A\to A^2\to I\to 0,
$$
where
$$
1\mapsto (-g,f),
\qquad
(a,b)\mapsto af+bg.
$$
Locally at a closed point where $I$ is not principal, this is the minimal free
resolution of $I$.

The chosen generators $f,g$ define a rational map
$$
X\bir \PP^1,
\qquad
y\mapsto (f(y):g(y)).
$$
They also give a closed embedding of the blowup:
$$
B_I X=\Proj \bigoplus_{d\geq 0} I^d
\hookrightarrow
X\times \PP^1.
$$
If $[u:v]$ are homogeneous coordinates on $\PP^1$, the image is cut out by
the equation
$$
gu-fv=0.
$$
Equivalently, $A\to A^2$ gives a section of
$\mathcal O_{X\times \PP^1}(1)$, and its zero divisor is precisely $B_I X$.

Thus, for a two-generated ideal on a smooth surface with coprime generators,
the minimal free resolution gives the equation of the blowup inside
$X\times \PP^1$. For ideals with more generators, a chosen surjection
$$
A^r\to I
$$
still gives an embedding
$$
B_I X\hookrightarrow X\times \PP^{r-1},
$$
but the equations of this embedding are no longer read off from a single
syzygy. They are governed by the Rees algebra of $I$, and the relation with a
minimal free resolution is more subtle.

%%%%%%%%%%%%%%%%%%%%%%%%%%%%%%
\subsection{Pullback of divisors via rational maps}

Let $X$ be a normal variety and let $f\colon X\bir Z$ be a rational map to a
projective variety. Let $D$ be a $\Q$-Cartier divisor on $Z$, and assume that
the generic image of $f$ is not contained in $\Supp D$. We define $f^*D$ as a
Weil $\Q$-divisor on $X$ by restricting $f$ to its regular locus, pulling back
$D$ there, and taking the closure in $X$. Equivalently, if
$\pi\colon W\to X$ is a normal model such that the induced map
$q\colon W\bir Z$ is a morphism, then
$$
f^*D=\pi_*q^*D.
$$
Even when $D$ is Cartier, the divisor $f^*D$ need not be Cartier or
$\Q$-Cartier.

Now let $f\colon X\bir \PP^n$ be given by rational functions
$f_0,\dots,f_n$, and let $H$ be the hyperplane 
$$
\sum_{i=0}^n a_i t_i=0
$$
not containing the image of $f$. Set
$$
L=\max_i\{-\Div(f_i)\},
$$
where the maximum is taken coefficient-wise. Thus $L$ is the smallest Weil
divisor such that
$$
L+\Div(f_i)\geq 0
$$
for all $i$. The functions $f_i$ define sections
$$
s_i\in H^0(X,\mathcal O_X(L)),
$$
where $\mathcal O_X(L)$ is understood as a reflexive sheaf, and
$$
\Div(s_i)=L+\Div(f_i).
$$
By the minimality of $L$, the effective divisors $L+\Div(f_i)$ have no common
component. Hence, on a big open set $U\subset X$, the sections $s_i$ are
base-point-free and define $f$. On $U$ we have
$$
(f|_U)^*\mathcal O_{\PP^n}(1)\simeq \mathcal O_U(L|_U),
$$
and therefore
$$
f^*H|_U
=
\Div (\sum_{i=0}^n a_i s_i)|_U
=
(L+\Div(\sum_{i=0}^n a_i f_i))|_U.
$$
Taking closures gives
$$
f^*H
=
L+\Div(\sum_{i=0}^n a_i f_i).
$$

Locally near a smooth, or locally factorial, point $x\in X$, one can multiply
the $f_i$ by a common rational function and then cancel their greatest common
divisor. For this primitive local presentation, $L=0$ near $x$, so the formula becomes
$$
f^*H=\Div(\sum_{i=0}^n a_i f_i).
$$

\begin{exa}
Let $X=\A^2$ with coordinates $u,v$, and let
$$
f\colon X\bir \PP^2, \ f=(u^3:u^2v:uv^2).
$$
For $H=\{t_0+t_1+t_2=0\}$, the naive expression
$u^3+u^2v+uv^2$ has the extra component $\{u=0\}$. Here
$$
L=\max\{-3\Div(u),-2\Div(u)-\Div(v),-\Div(u)-2\Div(v)\}
=-\Div(u).
$$
Thus
$$
f^*H
=
-\Div(u)+\Div(u^3+u^2v+uv^2)
=
\Div(u^2+uv+v^2).
$$
\end{exa}

%%%%%%%%%%%%%%%%%%%%%%%%%%%%%
\subsection{First examples}

We present a few basic examples of rational map singularities.

\begin{exa}[Normalisation]
Let $X$ be a variety, and let
$$
\nu\colon X^\nu\to X
$$
be its normalisation. Consider the birational inverse
$$
f\colon X\bir X^\nu.
$$
The singularities of $f$ measure the failure of $X$ to be normal, or
equivalently the complexity of the normalisation map. More generally, if
$$
X^\nu\to Z\to X
$$
is an intermediate finite birational model, then the birational inverse
$$
X\bir Z
$$
measures the corresponding partial normalisation of $X$. We will study this in more detail in the section on curves.
\end{exa}

\begin{exa}[Projection from a smooth point and the blowup]
Let $X$ be a smooth variety, let $x\in X$ be a closed point, and let
$t_1,\dots,t_n$ be a regular system of parameters at $x$. Locally near $x$,
these parameters define a rational map
$$
f\colon X\bir \PP^{n-1},
\qquad
f=(t_1:\dots:t_n).
$$
The base ideal of $f$ is the maximal ideal
$$
m_x=(t_1,\dots,t_n).
$$
Hence the graph of $f$ is the blowup
$$
\pi\colon Y=B_xX\to X.
$$
The morphism $\pi$ is an isomorphism away from $x$, so the singularity of $f$
is concentrated at $x$. Under the embedding
$$
Y\subset X\times \PP^{n-1},
$$
the polarisation induced from $\PP^{n-1}$ is the tautological line bundle of
the blowup:
$$
\mathcal O_Y(1)=\mathcal O_Y(-E),
$$
where $E$ is the exceptional divisor. This is the simplest map singularity on a smooth variety.
\end{exa}

\begin{exa}[A map with non-normal graph]
Consider
$$
f\colon \A^2\bir \PP^1,
\qquad
f=(x^2:y^n),
$$
with $n\geq 2$. Its graph is the hypersurface
$$
y^n u-x^2v=0
\quad\subset\quad
\A^2\times \PP^1,
$$
where $[u:v]$ are coordinates on $\PP^1$. This graph is non-normal
along the fibre over the origin $o$. It is more convenient to pass to the normalised graph $Y^\nu$ to measure the singularity of $f$.

For a general point $H\subset \PP^1$,
$$
f^*H=\Div(ax^2+by^n),
$$
so the multiplicity 
$$
\mu_of^*H=2.
$$
We will see later (Theorem \ref{t-norm-graph-fibre-degree-mult-hyperplane-pullback}) that 
$$
\deg\mathcal O_{Y^\nu}(1)|_{Y^\nu_o}=2 
$$
which in particular shows that the singularity of $f$ at the origin $o$ is not linear type.
\end{exa}

\begin{exa}[Maps with same normalised graph]
Consider the rational maps 
$$
f,h\colon \A^2\bir \PP^1, \ \ g\colon \A^2\bir \PP^2
$$ 
given by 
$$
f=(x:y), \ h=(x^2:y^2), \ g=(x^2:xy:y^2).
$$  
Denote the corresponding graphs by 
$$
Y_f, Y_h, Y_g,
$$ 
respectively. The maps are all {singular} at the origin $o$. Let $Y\to X$ be the blowup of $o$ with exceptional divisor $E$. Note that $E$ is the fibre of $Y$ over $o$. 

Then 
 $$
 Y_f=Y, \  \mathcal{O}_{Y_f}(1) = \mathcal{O}_Y(-E), \ \deg \mathcal{O}_{Y_f}(1)|_{E}=1
 $$
and  
 $$
 Y_g=Y, \ \mathcal{O}_{Y_g}(1) = \mathcal{O}_Y(-2E), \ \deg \mathcal{O}_{Y_g}(1)|_{E}=2,
 $$
 but $Y_h$ is not normal and  
 $$
 Y=Y_h^\nu, \ \mathcal{O}_{Y_h^\nu}(1) = \mathcal{O}_Y(-2E), \ \deg \mathcal{O}_{Y^\nu_h}(1)|_{E}=2.
$$
So the normalised graph can distinguish $f,g$ but not $g,h$. Note that the equality $Y=Y_h^\nu$ follows from the fact that $h$ is the composition of $g$ and the projection $\PP^2\bir \PP^1$ onto the first and last coordinates inducing a birational morphism $Y=Y_g\to Y_h$.
\end{exa}

\begin{exa}[Analogues of hypersurface singularities]\label{exa-along-hypersurface-sing}
The analogues of hypersurface singularities in the context of rational maps
are the singularities of maps
$$
f\colon X\bir \PP^1
$$
from smooth varieties. Indeed, after choosing a primitive local presentation
$f=(f_0:f_1)$, the graph is the blowup of the two-generated base ideal
$I=(f_0,f_1)$ and embeds into $X\times \PP^1$ as a divisor. 

In this sense, rational maps to $\PP^1$ play the role of hypersurfaces among
map singularities. It should therefore be useful to look for analogues of
classical invariants and methods from hypersurface singularity theory, such as
Milnor numbers, in this setting.
\end{exa}

\begin{exa}[Systematic construction of examples]
Let $X$ be an affine variety. One can construct examples of rational maps by choosing a closed
subvariety
$$
Y\subset X\times \PP^n
$$
such that the projection $Y\to X$ is birational. Then the second projection
induces a rational map
$$
f\colon X\bir \PP^n
$$
whose graph is $Y$. This gives a systematic way to produce examples of rational maps and their graphs.
\end{exa}

%%%%%%%%%%%%%%%%%%%%%%%%%%%%%%
%%%%%%%%%%%%%%%%%%%%%%%%%%%%%%
\section{\bf Maps on curves and multiplicity}

In this section we investigate singularities of rational maps on curves. We start with some basic examples.

%%%%%%%%%%%%%%%%%%%%%%%%%%%%%%
\subsection{Rational functions on curves}

Although this paper mainly concerns singularities of maps to projective varieties, we occasionally recall examples with non-projective targets.

\begin{exa}[Rational functions]
For a variety $X$, a rational function $f\in k(X)$ is the same as a rational map
$$
f\colon X \bir \A^1.
$$
The {singular} locus of $f$ can have any positive codimension. If $X$ is normal, then this locus is either empty or has pure codimension one. For instance, if $X=\A^n$ and
$$
f=\frac{g}{h},
$$
where $g,h$ have no common factor, then the {singular} locus is $V(h)$. The degrees of the irreducible factors of $h$, with multiplicities, give a natural measure of the {singularities} of $f$.
\end{exa}

\begin{exa}[Dedekind domains]
Let $A$ be a Dedekind domain with fraction field $K$. A rational map
$$
\Spec A \bir \Spec A[t]
$$
over $A$ is determined by an $A$-homomorphism $A[t]\to K$, equivalently by the image $t\mapsto f\in K$.

For each nonzero prime ideal $p\subset A$, let ${\rm ord}_{p}$ be the corresponding valuation. The singular locus of the map is
$$
\{p\mid {\rm ord}_{p}(f)<0\},
$$
that is, the support of the pole divisor
$$
\operatorname{Pole}(f)
=
\sum_{p}
\max\{-{\rm ord}_{p}(f),0\} p.
$$
Thus the pole divisor measures the singularities of the map.
\end{exa}

\begin{exa}[Integers]
Giving a rational map
$$
\Spec \Z \bir \Spec \Z[t]
$$
over $\Z$ is the same as giving a rational number $f\in \Q$. If $f=a/b$ with
$(a,b)=1$, then the {singular} locus is supported at the primes dividing
$b$. Moreover, the order of {singularity} at such a prime $p$ is measured
by ${\rm ord}_p(b)=-{\rm ord}_p(f)$: the larger the power of $p$ in the denominator, the
more singular $f$ is at $p$.
\end{exa}

%%%%%%%%%%%%%%%%%%%%%%%%%%%%%%
\subsection{Multiplicity of curves and normalisation}

We will now investigate singularities of the inverse of normalisation $X^\nu\to X$ of a curve.
We recall the Hilbert--Samuel multiplicity; see for example \cite[\S\S13--14]{Matsumura-CRT}.

\begin{defn}[Hilbert--Samuel multiplicity]
Let $(A,m)$ be a Noetherian local ring, let $I\subset A$ be an $m$-primary
ideal, and let $M$ be a finitely generated $A$-module of dimension $d$. For $r\gg 0$,
$$
H_{I,M}(r):=\operatorname{length}_A(M/I^{r+1}M)
$$
agrees with a polynomial $P_{I,M}(r)$ of degree $d$. The Hilbert--Samuel
multiplicity of $M$ with respect to $I$ is
$$
e_I(M):=d!\cdot \text{the leading coefficient of }P_{I,M}.
$$
In particular, if $A$ is one-dimensional,
then
$$
\operatorname{length}_A(M/m^{r+1}M)=e_m(M)r+O(1).
$$

When $X$ is a Noetherian scheme and $x\in X$ is a closed point, we write $\mu_x X:=e_{m_x}(A)$, where
$A=\mathcal O_{X,x}$ and $m_x$ is the maximal ideal, and call it the multiplicity of $X$ at $x$.
\end{defn}

The following fact identifies the multiplicity of a reduced curve singularity with the degree of the scheme-theoretic fibre of its normalisation; cf. \cite[\S1]{Orecchia-ordinary-curves}.

\begin{lem}[Multiplicity and the normalisation fibre for curves]\label{t-mult-normalisation-fibre-curves}
Let $X$ be a reduced (not necessarily irreducible) curve over the ground field $k$, and let
$$
\nu:X^\nu \longrightarrow X
$$
be its normalisation. Let $x\in X$ be a closed point, and let
$$
X^\nu_x=X^\nu\times_X \Spec k(x)
$$
be the scheme-theoretic fibre. Then the multiplicity 
$$
\mu_x X=\deg X^\nu_x.
$$
\end{lem}

\begin{proof}
Set
$$
A=\mathcal O_{X,x},
\qquad
B=(\nu_*\mathcal O_{X^\nu})_x.
$$
Since $\nu$ is finite and birational, $B/A$ has finite length. Hence
$$
e_{m_x}(A)=e_{m_x}(B),
$$
because adding a finite-length module changes the Hilbert--Samuel function only
by a constant.

For each $x'$ mapping to $x$ we can write
$$
m_x\mathcal O_{X^\nu,x'}=m_{x'}^{p_{x'}}
$$
for some number $p_{x'}$. 
We then obtain
$$
\operatorname{length}(B/m_x^{r+1}B)
=
(r+1)\sum_{x'\mapsto x}p_{x'}.
$$
 Therefore
$$
e_{m_x}(B)=\sum_{x'\mapsto x}p_{x'}.
$$
Moreover,
$$
\operatorname{length}(B/m_xB)
=
\sum_{x'\mapsto x}p_{x'}
=
e_{m_x}(B).
$$

On the other hand,
$$
B/m_xB
=
(\nu_*\mathcal O_{X^\nu})_x
\otimes_{\mathcal O_{X,x}}
k(x)
$$
is the coordinate ring of the scheme-theoretic fibre $X^\nu_x$. Hence
$$
\mu_xX=e_{m_x}(A)=e_{m_x}(B)=\operatorname{length}(B/m_xB)
=
\deg X^\nu_x.
$$
\end{proof}

This can also be proved in a more geometric way by assuming $X$ affine and interpreting $\mu_xX$ as the local intersection number $(X\cdot H)_x$ for a general hyperplane $H$ through $x$, and resolving $X$ through a sequence of point blowups.

%%%%%%%%%%%%%%%%%%%%%%%%%%%%%%
%%%%%%%%%%%%%%%%%%%%%%%%%%%%%%
\section{\bf The normalised graph and associated invariants}

In this section we study fibres of the normalised graph of a rational map and associated invariants.
Recall the {normalised polarised graph} of a rational map $f\colon X\bir \PP^n$ from a variety $X$:
$$
\xymatrix{
& Y^\nu,\mathcal O_{Y^\nu}(1) \ar[dl]_{p} \ar[dr]^{q} & \\
X && \PP^n .
}
$$
Fix a point \(x\in X\). The scheme-theoretic fibre of the normalised graph over \(x\) is denoted 
$$
Y^\nu_x.
$$
Restricting the polarisation gives
$$
\mathcal O_{Y^\nu_x}(1):=\mathcal O_{Y^\nu}(1)|_{Y^\nu_x},
$$
and restricting \(q\) gives a morphism
$$
q_x\colon Y^\nu_x\to \PP^n.
$$
Thus the local object attached to \(f\) at \(x\) is the \emph{polarised normalised graph fibre}  
packaged into the diagram
$$
\xymatrix{
& Y^\nu_x,\mathcal O_{Y^\nu_x}(1) \ar[dl] \ar[dr]^{q_x} & \\
\Spec k(x) && \PP^n .
}
$$

The geometric complexity of $Y^\nu_x\to \PP^n$ gives a direct way to measure the complexity of the rational map $f$ at $x$. For example, linear type maps will have the simplest kind of $Y^\nu_x\to \PP^n$, which is a closed embedding of degree one. 

In general, we attach algebraic and geometric structures to $Y^\nu_x,\mathcal O_{Y^\nu_x}(1)\to \PP^n$. The simplest invariant is the \emph{normalised graph fibre degree}  
$$
\delta_x(f):=\deg \mathcal O_{Y^\nu}(1)|_{Y^\nu_x}=\mathcal O_{Y^\nu_x}(1)^{\dim Y^\nu_x}
$$ 
where the right hand side denotes the top self-intersection number of the polarisation.

A more refined invariant is the Hilbert polynomial $\Phi_{x,f}$ given by 
$$
\Phi_{x,f}(m)=\chi(Y^\nu_x,\mathcal O_{Y^\nu_x}(m))=\frac{\delta_x(f)}{r!}m^r+O(m^{r-1})
$$
where $r=\dim Y^\nu_x$.

Another important invariant is the \emph{very ampleness index} defined to be the smallest natural number $m$ so that 
$\mathcal O_{Y^\nu}(m)$ is very ample over a neighbourhood of $x$.

%%%%%%%%%%%%%%%%%%%%%%%%%%%%%%
\subsection{Normalised fibre degree spectrum}
Let
$$
f:X\dashrightarrow \mathbb P^n
$$
be a rational map from a variety and $x\in X$ a closed point. 
The polarised normalised graph fibre gives a collection of numbers, namely, \emph{the degree spectrum of $f$ at $x$} by
$$
\operatorname{DSpec}_x(f)
:=
\{
\deg_{\mathcal O_{Y^\nu}(1)} A\}
$$
where $A\ge 0$ runs through the pure-dimensional cycles on $Y^\nu$ supported on irreducible components of $Y^\nu_x$, and 
$$
\deg_{\mathcal O_{Y^\nu}(1)} A:=\mathcal O_{Y^\nu}(1)^{\dim A}\cdot A.
$$
In particular, the degree 
$$
\delta_x(f)=\deg \mathcal O_{Y^\nu}(1)|_{Y^\nu_x}
$$
is one of the numbers in the degree spectrum.

In the case of smooth surfaces $X$ we will see later that the degree spectrum is closely related to canonical thresholds in a sequence of smooth blowups resolving $f$.

%%%%%%%%%%%%%%%%%%%%%%%%%%%%%%
\subsection{Multiplicity of general hyperplane pullback}

\begin{thm}[Normalised graph fibre degree and general hyperplane pullback]\label{t-norm-graph-fibre-degree-mult-hyperplane-pullback}
Let $X$ be a normal surface and let
$$
f\colon X\dashrightarrow \mathbb P^n
$$
be a rational map singular at a closed point $x\in X$. Let
$$
p\colon Y^\nu\longrightarrow X,
\qquad
q\colon Y^\nu\longrightarrow \mathbb P^n
$$
be the normalised graph of $f$.
 
Let $H\subset \mathbb P^n$ be a general hyperplane, and 
$$
V:=f^*H\subset X.
$$
Then
$$
\delta_x(f)=\mu_x V.
$$
\end{thm}

\begin{proof}
Since $H$ is a general hyperplane and $Y^\nu$ is normal, $D:=q^*(H)$ is a smooth smooth, possibly disconnected, divisor and $V=p_*D$ is reduced. Moreover, the induced morphism 
$$
p|_D\colon D\longrightarrow V
$$
is the normalisation of $V$. 

By definition, 
$$
\delta_x(f)=\deg \mathcal O_{Y^\nu}(1)|_{Y^\nu_x}=\deg D|_{Y^\nu_x}.
$$
This in turn is equal to the scheme-theoretic intersection length   
$$
\operatorname{length}(D\cap Y^\nu_x).
$$
However, the intersection $D\cap Y^\nu_x$ coincides with the fibre $D_x$ of $D\to V$ over $x$.

On the other hand, by Theorem \ref{t-mult-normalisation-fibre-curves}, 
$$
\mu_xV=\deg D_x,
$$
hence 
$$
\mu_xV=\deg D_x=\operatorname{length}(D\cap Y^\nu_x)=\deg D|_{Y^\nu_x}=\delta_x(f). 
$$
\end{proof}

When $f$ is represented locally by a base ideal $I$, the number $\delta_x(f)$ is the first polar multiplicity $m_1(I,X)$ at $x$. Thus the theorem is closely related to the description of the first polar multiplicity as the multiplicity of a general polar divisor; see \cite[\S2 and \S3]{Gaffney-Gassler}. The formulation above is intrinsic to the rational map and applies in particular when $X$ is normal singular and $f^*H$ need not be Cartier.

%%%%%%%%%%%%%%%%%%%%%%%%%%%%%%
\subsection{Toric surface maps}
We combinatorially describe the normalised graph fibre degree in the case of toric maps from surfaces to $\PP^1$.
For basics of toric geometry used in this paper, see \cite{Cox-Little-Schenck}.

\begin{thm}[Normalised graph fibre degree for toric surface maps]
\label{thm:toric-normalised-graph-fibre-degree}
Let $N=\Z^2$ and $N'=\Z$. Let
$$
\sigma=\langle v_0,v_1\rangle\subset N_{\R}
$$
be a two-dimensional strongly convex rational polyhedral cone, where $v_0,v_1$ are primitive, and let
$$
X=X_\sigma,
\qquad
x=x_\sigma
$$
be the corresponding affine toric surface and its closed torus-fixed point.

Let
$$
\alpha\colon N\longrightarrow N'
$$
be a nonzero homomorphism such that $\alpha(v_0)$ and $\alpha(v_1)$ lie on opposite sides of the origin. Then $\alpha$ induces a toric rational map
$$
f\colon X\bir\PP^1
$$
which is singular at $x$.

Let $w\in N$ be the primitive generator of $\ker\alpha$ lying in the interior of $\sigma$, and define
$$
\mu_\sigma(w)
:=
\min\left\{
\langle m,w\rangle>0
\mid
m\in\sigma^\vee\cap M
\right\}.
$$
Also set
$$
d(\alpha):=[N':\alpha(N)].
$$
Let
$$
p\colon Y^\nu\longrightarrow X,
\qquad
q\colon Y^\nu\longrightarrow\PP^1
$$
be the normalised graph of $f$. Then $Y^\nu$ is the toric surface obtained by subdividing $\sigma$ along the ray $\R_{\ge 0}w$. If
$$
E=D_w
$$
is the corresponding exceptional divisor, then  
$$
[Y^\nu_x]=\mu_\sigma(w)E,
$$
where the left hand side denotes the associated cycle of the scheme-theoretic fibre. Moreover,
$$
E\simeq\PP^1
\ \mbox{and} \
\deg\mathcal{O}_{Y^\nu}(1)|_E=d(\alpha).
$$
Consequently,
$$
\delta_x(f)
=
\mu_\sigma(w)d(\alpha).
$$
\end{thm}

\begin{proof}
The fan of $\PP^1$ in $N'_{\R}$ consists of the origin and the two rays on opposite sides of the origin. Since $\alpha(v_0)$ and $\alpha(v_1)$ lie on opposite sides, the inverse images of these two rays subdivide $\sigma$ along
$$
\ker(\alpha)_{\R}\cap\sigma=\R_{\ge 0}w.
$$
Therefore the fan of the normalised graph is the star subdivision of $\sigma$ along the ray $\R_{\ge 0}w$. Its two maximal cones are
$$
\langle v_0,w\rangle
\qquad\mbox{and}\qquad
\langle w,v_1\rangle.
$$
The exceptional locus of
$$
p\colon Y^\nu\longrightarrow X
$$
is the torus-invariant prime divisor
$$
E=D_w.
$$

We first calculate the multiplicity of $E$ in the fibre over $x$. By \cite[Corollary 1.3.3]{Cox-Little-Schenck}, the maximal ideal of the torus-fixed point $x$ is
$$
\mathfrak m_x
=
\left(
\chi^m
\mid
0\neq m\in\sigma^\vee\cap M
\right).
$$
For every $m\in M$, 
$$
\operatorname{ord}_E(\chi^m)=\langle m,w\rangle.
$$
Since $w\in\operatorname{Int}(\sigma)$, we have
$$
\langle m,w\rangle>0
$$
for every nonzero $m\in\sigma^\vee\cap M$. It follows that
$$
\operatorname{ord}_E(\mathfrak m_x)
=
\min\left\{
\langle m,w\rangle>0
\mid
m\in\sigma^\vee\cap M
\right\}
=
\mu_\sigma(w).
$$

Let $\eta_E$ be the generic point of $E$. Since $Y^\nu$ is normal,
$$
\mathcal{O}_{Y^\nu,\eta_E}
$$
is a discrete valuation ring. If $t$ is a uniformising parameter, then
$$
\mathfrak m_x\mathcal{O}_{Y^\nu,\eta_E}
=
(t^{\mu_\sigma(w)}).
$$
Thus the generic length of the fibre along $E$ is $\mu_\sigma(w)$. Since the support of the fibre is $E$, its fundamental cycle is
$$
[Y^\nu_x]=\mu_\sigma(w)E.
$$

We next calculate the degree of the polarisation on $E$. The fan of $E$ is the star of the ray $\R_{\ge 0}w$ in the fan of $Y^\nu$. Equivalently, it is the fan in
$$
\overline N:=N/\Z w
$$
whose two rays are generated by the images of $v_0$ and $v_1$. This is the complete fan of a toric curve, and hence
$$
E\simeq\PP^1.
$$

Since $\ker\alpha=\Z w$, the homomorphism $\alpha$ factors as
$$
N\longrightarrow\overline N
\xrightarrow{\overline\alpha}
N',
$$
where $\overline\alpha$ is injective and
$$
\overline\alpha(\overline N)=\alpha(N).
$$
The restriction
$$
q|_E\colon E\longrightarrow\PP^1
$$
is the finite toric morphism induced by $\overline\alpha$. Its degree is therefore
$$
\deg(q|_E)
=
[N':\overline\alpha(\overline N)]
=
[N':\alpha(N)]
=
d(\alpha).
$$
Since
$$
\mathcal{O}_{Y^\nu}(1)=q^*\mathcal{O}_{\PP^1}(1),
$$
we obtain
$$
\deg\mathcal{O}_{Y^\nu}(1)|_E
=
\deg(q|_E)
=
d(\alpha).
$$

Finally,
$$
\delta_x(f)
=
\deg_{\mathcal{O}_{Y^\nu}(1)}Y^\nu_x
=
\mathcal{O}_{Y^\nu}(1)\cdot[Y^\nu_x].
$$
Using the preceding calculations gives
$$
\delta_x(f)
=
\mu_\sigma(w)
\deg\mathcal{O}_{Y^\nu}(1)|_E
=
\mu_\sigma(w)d(\alpha).
$$
\end{proof} 

%%%%%%%%%%%%%%%%%%%%%%%%%%%%%%
\subsection{Examples}

\begin{exa}[Degree of maximal ideal maps and multiplicity]\label{e-maximal ideal maps}
Let $x\in X$ be a normal singularity of dimension $d$ and let $m_x$ be the maximal ideal of $x$. Choose generators
$f_0,\ldots,f_n$ of $m_x$, and let
$$
f\colon X\dashrightarrow \PP^n,
\qquad f=(f_0:\cdots:f_n),
$$
be the corresponding rational map. Then the blowup
$$
p\colon Y=\operatorname{B}_{m_x}X\longrightarrow X
$$
is precisely the graph of $f$.

We have 
$$
Y=\operatorname{Proj}\bigoplus_{i\geq 0}m_x^i,
\qquad
m_x\mathcal O_Y=\mathcal O_Y(1).
$$
The scheme-theoretic fibre $Y_x$ over $x$ is the effective Cartier divisor
defined by $m_x\mathcal O_Y$. Moreover,
$$
Y_x
=
\operatorname{Proj}
\bigoplus_{i\geq 0}\frac{m_x^i}{m_x^{i+1}}.
$$
Recall that the Hilbert--Samuel multiplicity is the degree of the exceptional fibre of the blowup of the maximal ideal; see \cite[\S4.3]{Fulton-Intersection-Theory}.
For $i\gg 0$, 
$$
\dim_k \frac{m_x^i}{m_x^{i+1}}=\mathcal{X}(Y_x,O_Y(i)|_{Y_x}).
$$
The right hand side is the Hilbert polynomial of $Y_x$ with respect to $O_Y(1)|_{Y_x}$, which is a polynomial of degree $d-1$ with leading coefficient 
$$
\frac{\deg\mathcal O_Y(1)|_{Y_x}}{(d-1)!}.
$$

On the other hand, 
$$
\sum_{0\le j\le i}\dim \frac{m_x^j}{m_x^{j+1}}
$$ 
is the Hilbert--Samuel polynomial in $i$ of degree $d$ with leading coefficient $\frac{\mu_xX}{d!}$.
From these one deduces that 
$$
\mu_xX
=
\deg\mathcal O_Y(1)|_{Y_x}.
$$

Now let
$$
\nu\colon Y^\nu\longrightarrow Y
$$
be the normalization. The pullback of $m_x\mathcal O_Y$ is again an
invertible ideal sheaf. Hence if $m_x\mathcal O_Y$ defines the effective Cartier divisor $E$, then its pullback defines the effective exceptional Cartier divisor $\nu^*E$ so that 
$$
m_x\mathcal O_{Y^\nu}
=
\nu^*\mathcal O_Y(1)
=
\mathcal O_{Y^\nu}(1)
=
\mathcal O_{Y^\nu}(-\nu^*E).
$$
The scheme-theoretic fibre of $Y^\nu\to X$ over $x$ is precisely $\nu^*E$.

Moreover, 
$$
\deg\mathcal O_Y(1)|_{Y_x}=(-E|_E)^{d-1}=(-1)^{d-1}E^d
$$
and 
$$
\deg \mathcal O_{Y^\nu}(1)|_{Y^\nu_x}=(-\nu^*E|_{\nu^*E})^{d-1}=(-1)^{d-1}\nu^*E^d.
$$
Since $\nu$ is birational, it follows that the normalised graph fibre degree coincides with the multiplicity:
$$
\delta_x(f)=\deg \mathcal O_{Y^\nu}(1)|_{Y^\nu_x}=(-1)^{d-1}E^d=(-1)^{d-1}\nu^*E^d
=
\deg\mathcal O_Y(1)|_{Y_x}
=
\mu_xX.
$$
\end{exa}
\medskip

\begin{exa}[Rational maps from a cone over an elliptic curve]
Let
$$
X\subset\mathbb A^3
$$
be the affine cone defined by
$$
y^2z-x^3-xz^2=0,
$$
and let $o=(0,0,0)$ be its vertex. Consider the rational map
$$
f\colon X\dashrightarrow\mathbb P^1,
\qquad
f=(x:z)=(y^2:x^2+z^2).
$$
This map is singular at $o$. 

We can use Theorem \ref{t-norm-graph-fibre-degree-mult-hyperplane-pullback} to calculate $\delta_o(f)$. Take a general hyperplane $u-bv=0$ where $u,v$ are the coordinates on $\PP^1$. Then $f^*H$ can be derived from intersecting $X$ with $x-bz=0$, that is,  
$$
y^2z-b^3z^3-bz^3=z(y^2-b^3z^2-bz^2)=0.
$$
Since the factor $z$ is independent of $b$, $f^*H$ is given by the equation 
$$
y^2-b^3z^2-bz^2=0
$$
in $\A^2$ with coordinates $y,z$. This has multiplicity $2$ at $(0,0)$, so we deduce that $\delta_o(f)=2$.

We can also calculate $\delta_o(f)$ directly. Let

$$
E\subset \mathbb P^2 \qquad \mbox{defined by} \qquad {y^2z-x^3-xz^2=0}
$$
be the smooth plane cubic over which $X$ is the affine cone, and let
$$
\pi\colon W\to X
$$
be the blowup of the vertex $o$. Then $W$ is smooth, the exceptional divisor is $E$, and $E^2=-3$.

Let $I=(x,z)\subset k[X]$. The graph $Y$ of $f=(x:z)$ is the blowup of $I$.
Intersecting $X$ with $x=0$ in $\A^3$ gives $y^2z=0$. Thus on $X$,
$$
\Div(x)=2L+M,
$$
where $L$ is given by ${x=y=0}$ and $M$ by ${x=z=0}$. Similarly, intersecting $X$ with $z=0$ gives $x^3=0$, and hence
$$
\Div(z)=3M.
$$

Let $L'$ and $M'$ be the birational transforms of $L$ and $M$ on $W$. Since $x$ and $z$ vanish to order one at the vertex, on $W$ we have 
$$
\Div(\pi^*x)=E+2L'+M'
$$
and
$$
\Div(\pi^*z)=E+3M'.
$$
Their common divisorial part is $E+M'$. Since $L'$ and $M'$ are disjoint,
$$
I\mathcal O_W=\mathcal O_W(-E-M').
$$
Thus $I\mathcal O_W$ is invertible and we derive two of its sections from $\pi^*x,\pi^*z$ inducing a finite birational morphism $W\to Y$ over $X$. This morphism is an isomorphism away from $E$, and its restriction to $E$ is a degree $2$ morphism $E\to \PP^1$. In particular, since $W$ is smooth,
$$
Y^\nu=W.
$$

Moreover,
$$
\mathcal O_{Y^\nu}(1)=\mathcal O_W(-E-M').
$$
Since $M'\cdot E=1$, we have
$$
\deg\mathcal O_{Y^\nu}(1)|_E=(-E-M')\cdot E=3-1=2.
$$
Therefore,
$$
\delta_o(f)=2
$$
keeping in mind that $E=Y^\nu_o$ is the fibre of $W=Y^\nu\to X$ over $o$.
\end{exa}

%%%%%%%%%%%%%%%%%%%%%%%%%%%%%%
%%%%%%%%%%%%%%%%%%%%%%%%%%%%%%
\section{\bf Thresholds of rational maps on klt varieties}

In this section we study rational maps $f\colon X\bir \PP^n$ from klt varieties $X$. A useful invariant of $f$ in this context is given by the log canonical threshold of $f^*H$ (in a suitable sense) for a hyperlpane section $H\subset \PP^n$. A larger threshold indicates simpler singularities for $f$. However, since the threshold is also affected by the singularities of $X$ itself, a small threshold does not necessarily mean that map is more singular.  
In latter sections we will also consider canonical thresholds.

We will use the language of generalised pairs in this section. For a general introduction, see \cite{Birkar-generalised-pairs}.

%%%%%%%%%%%%%%%%%%%%%%%%%%%%%%
\subsection{The generalised lc threshold of a hyperplane pullback}

Let $x\in X$ be a klt variety and let
$$
f\colon X\dashrightarrow \mathbb P^n
$$
be a rational map. Let
$$
p\colon Y^\nu\to X,\qquad q\colon Y^\nu\to \mathbb P^n
$$
be the normalised graph of $f$ with its associated morphisms. Let $H\subset \mathbb P^n$ be a hyperplane not containing the image of $X$ (e.g. we can take $H$ to be general), and assume that the Weil divisor $f^*H$ is $\mathbb Q$-Cartier near $x$. This assumption is automatic when $X$ is a klt surface.

For any $t\ge 0$ consider $(X,t f^*H)$ as a \emph{generalised pair} with nef part $tq^*H$. We define the \emph{generalised lc threshold}
$$
\lambda_x(f):=\sup\{t\ge 0\mid (X,t f^*H)\text{ is generalised lc at }x\}.
$$
More precisely, on $Y^\nu$ we can write
$$
p^*f^*H=q^*H+E_f,
$$
where $E_f\ge 0$ is exceptional over $X$, by the negativity lemma.

For $t\geq 0$, define $B_t$ on $Y^\nu$ by
$$
K_{Y^\nu}+B_t=p^*K_X+tE_f.
$$
Equivalently,
$$
K_{Y^\nu}+B_t+tq^*H=p^*(K_X+t f^*H).
$$
Then
$$
\lambda_x(f):=\sup\{t\geq 0\mid (Y^\nu,B_t)\text{ is sub-lc over }x\}.
$$

The definition of $\lambda_x(f)$ actually depends only on the $\mathbb Q$-linear equivalence class of $H$, so in particular it is independent of the choice of hyperplane $H$.

The above threshold can also be defined using usual pairs (rather than generalised pairs) making use of movable lc thresholds. But the generalised language is more convenient and standard nowadays.

\begin{rem}
Now let
$$
\rho\colon X'\to X
$$
be the local index-one cover at $x$, and let
$$
f'\colon X'\dashrightarrow \mathbb P^n
$$
be the induced rational map. Since $\rho$ is finite and also étale in codimension one,
$$
K_{X'}=\rho^*K_X
$$
and
$$
f'^*H=\rho^*f^*H.
$$
Moreover, $\rho^{-1}(x)$ is a unique point $x'$  (where $\rho^{-1}$ means set-theoretic inverse image), and we have
$$
\lambda_x(f)=\lambda_{x'}(f').
$$
Hence one may pass to the index-one cover and work in the Gorenstein case. In particular, in dimension two the index-one cover is smooth or has a Du Val singularity. 
\end{rem}

\begin{rem}
It is worth noting that if $x\in X$ is smooth and $I\subset \mathcal{O}_{X,x}$ is the primitive local base ideal of $f$, then
$$
\lambda_x(f)=\operatorname{lct}_x(I).
$$
Indeed, if
$$
\phi\colon W\to X
$$
is a log resolution of $I$ on which $f$ is resolved, and
$$
I\mathcal{O}_W=\mathcal{O}_W(-F),
$$
then for a general hyperplane $H\subset \mathbb P^n$ we have
$$
\phi^*f^*H=h^*H+F,
$$
where $h\colon W\to\mathbb P^n$ is the induced morphism. See \cite{Lazarsfeld-Positivity-II} and \cite[Example 4.5(6)]{Birkar-generalised-pairs} for relevant discussions.
\end{rem}

%%%%%%%%%%%%%%%%%%%%%%%%%%%%%%
\subsection{The ascending chain condition}

One guiding principle in the study of singularities of rational maps is that
there should not exist an infinite sequence of maps
$$
f_i\colon X_i\dashrightarrow\PP^n
$$
of fixed dimension whose singularities strictly improve. The precise meaning
of this statement depends on the invariant used to measure the singularities.

For the normalised graph fibre degree, this is immediate: the numbers
$\delta_{x_i}(f_i)$ are nonnegative integers, so they cannot form an infinite
strictly decreasing sequence.

For thresholds the statement is much subtler. Let the $X_i$ be klt varieties
of fixed dimension, let $x_i\in X_i$ be closed points, and assume that
$f_i^*H$ is $\Q$-Cartier near $x_i$. 
By the ACC for generalised lc thresholds \cite{Birkar-Zhang,HMX-ACC}, the 
thresholds
$$
\lambda_{x_i}(f_i)
$$
satisfy the ascending chain condition (ACC). In particular, there is no infinite
sequence of such rational maps for which
$$
\lambda_{x_1}(f_1)
<
\lambda_{x_2}(f_2)
<
\lambda_{x_3}(f_3)
<
\cdots.
$$
Thus, when larger thresholds are interpreted as milder singularities, the
singularities of the maps cannot improve indefinitely.

%%%%%%%%%%%%%%%%%%%%%%%%%%%%%%
\subsection{The toric case}

We treat the toric case using the notation introduced in the previous subsection.

\begin{prop}[The toric case]\label{p-rational-map-lc-threshold-toric-case}
Let $X=X_\sigma$ be a klt affine toric variety, let $x=x_\sigma$ be its closed torus-fixed point, and let
$$
f\colon X\dashrightarrow \mathbb P^n
$$
be a toric rational map. Assume that $f^*H$ is $\mathbb Q$-Cartier near $x$, where $H\subset\mathbb P^n$ is a hyperplane not containing the image of $X$.

Let $\Sigma$ be the fan of $Y^\nu$. Write
$$
E_f=\sum_{\rho\in\Sigma(1)}e_\rho D_\rho,
$$
where $D_\rho$ is the torus-invariant prime divisor corresponding to $\rho$. Then
$$
\lambda_x(f)=\min\{\frac{\psi_\sigma(v_\rho)}{e_\rho}\mid \rho\in\Sigma(1),\ e_\rho>0,\ x\in p(D_\rho)\}.
$$
Here $v_\rho$ is the primitive generator of $\rho$, and $\psi_\sigma$ is the support function of $K_X$ which calculates log discrepancies:
$$
a(D_v,X,0)=\psi_\sigma(v).
$$
If there is no ray with $e_\rho>0$ and $x\in p(D_\rho)$, then $\lambda_x(f)=+\infty$.
\end{prop}

\begin{proof}
Since $f$ is toric, its normalised graph $Y^\nu$ is toric and
$$
p\colon Y^\nu\to X
$$
is a toric birational morphism. choosing a toric hyperplane $H\subset \PP^n$, the divisor
$$
E_f=p^*f^*H-q^*H
$$
is the torus-invariant exceptional fixed part of $p^*f^*H$.

For $t\geq 0$, define the toric sub-pair $(Y^\nu,B_t)$ by
$$
K_{Y^\nu}+B_t=p^*K_X+tE_f.
$$
For every $\rho\in\Sigma(1)$, the coefficient of $D_\rho$ in $B_t$ is
$$
1-\psi_\sigma(v_\rho)+te_\rho,
$$
see \cite[\S11.4]{Cox-Little-Schenck}.
Thus the sub-lc condition along $D_\rho$ is
$$
te_\rho\leq\psi_\sigma(v_\rho).
$$

On the other hand, let $\Lambda$ be the toric boundary on $Y^\nu$, that is, the sum of the prime torus-invariant divisors on $Y^\nu$. If $B_t\le \Lambda$ for some $t$, then $K_{Y^\nu}+B_t$ is automatically sub-lc because $K_{Y^\nu}+\Lambda$ is lc. Therefore, $(X,tf^*H)$ is generalised lc iff $B_t\le \Lambda$. In particular, if $\lambda_x(f)<+\infty$, then $\lambda_x(f)$ is already calculated on $Y^\nu$ by the minimum in the statement of the proposition.  
\end{proof}

\begin{cor}[The toric surface case]\label{cor-rational-map-lc-threshold-toric-surface}
Let
$$
X=X_\sigma,
\qquad
\sigma=\langle u_1,u_2\rangle\subset N_{\mathbb R},
$$
be a klt affine toric surface, and let $x=x_\sigma$ be its torus-fixed
point. Let
$$
f\colon X\dashrightarrow\mathbb P^1
$$
be the nonconstant toric rational map induced by a homomorphism
$$
\alpha\colon N\longrightarrow\mathbb Z.
$$
Assume that $x$ is an indeterminacy point of $f$. 
Then
$$
\lambda_x(f)
=
\frac{1}{|\alpha(u_1)|}
+
\frac{1}{|\alpha(u_2)|}.
$$
\end{cor}

\begin{proof}
After interchanging $u_1$ and $u_2$, we have
$$
\alpha(u_1)>0>\alpha(u_2).
$$
Let $w\in N$ be the primitive generator of $\ker(\alpha)$ lying in the
interior of $\sigma$, and write
$$
a=\alpha(u_1)>0,
\qquad
b=-\alpha(u_2)>0.
$$
There are positive rational numbers $c_1,c_2$ such that
$$
w=c_1u_1+c_2u_2.
$$
Applying $\alpha$ gives
$$
ac_1=bc_2.
$$

The fan of $\mathbb P^1$ consists of the rays
$\mathbb R_{\geq 0}$ and $\mathbb R_{\leq 0}$. Its inverse image under
$\alpha$ subdivides $\sigma$ into the two cones
$$
\langle u_1,w\rangle
\qquad\text{and}\qquad
\langle w,u_2\rangle.
$$
Thus the fan of the normalised graph $Y^\nu$ is obtained by inserting the
ray $\mathbb R_{\geq 0}w$. Let $E_w$ be the corresponding exceptional
divisor.

Let $H\in\mathbb P^1$ be the point at infinity, i.e. the toric divisor corresponding to the ray $\mathbb R_{\leq 0}$.
On $Y^\nu$ we have
$$
q^*H=bD_{u_2},
$$
while on $X$ we have
$$
f^*H=bD^X_{u_2},
$$
where $D_{u_2},D^X_{u_2}$ denote the torus-invariant divisors corresponding to
$u_2$ on $Y^\nu$ and $X$, respectively.

The support function $\psi_{u_2}$ of $D^X_{u_2}$ satisfies
$$
\psi_{u_2}(u_1)=0,
\qquad
\psi_{u_2}(u_2)=-1.
$$
Since $w=c_1u_1+c_2u_2$, we obtain
$$
\psi_{u_2}(w)=-c_2.
$$
It follows that
$$
p^*f^*H=q^*H+bc_2E_w,
$$
and hence
$$
E_f=bc_2E_w.
$$

The toric log discrepancy function $\psi_\sigma$ satisfies
$$
\psi_\sigma(u_1)=\psi_\sigma(u_2)=1.
$$
Therefore
$$
a(E_w,X,0)=\psi_\sigma(w)=c_1+c_2.
$$

By the preceding proposition, the only ray contributing to the computation
of $\lambda_x(f)$ is the ray generated by $w$. Thus
$$
\lambda_x(f)=\frac{c_1+c_2}{bc_2}.
$$
Using $ac_1=bc_2$, we obtain
$$
\lambda_x(f)
=
\frac{1}{a}+\frac{1}{b}
=
\frac{1}{|\alpha(u_1)|}
+
\frac{1}{|\alpha(u_2)|}.
$$
\end{proof}

\begin{exa}[The map $(x^m:y^n)$ on $\A^2$]\label{exa-map-(xm:yn)-A2}
Let
$$
X=\A^2=\operatorname{Spec}k[x,y],
\qquad
o=(0,0),
$$
and consider the rational map
$$
f\colon X\dashrightarrow\PP^1,
\qquad
f=(x^m:y^n),
$$
where $m,n>0$.

The cone of $\A^2$ is
$$
\sigma=\langle u_1,u_2\rangle
=
\langle(1,0),(0,1)\rangle,
$$
and the homomorphism of lattices corresponding to $f$ is
$$
\alpha\colon N=\Z^2\longrightarrow\Z,
\qquad
\alpha(a,b)=ma-nb.
$$
Thus
$$
\alpha(u_1)=m,
\qquad
\alpha(u_2)=-n.
$$
In particular, $o$ is an indeterminacy point of $f$. By the toric surface
corollary,
$$
\lambda_o(f)
=
\frac{1}{|\alpha(u_1)|}
+
\frac{1}{|\alpha(u_2)|}
=
\frac{1}{m}+\frac{1}{n}.
$$

For example,
$$
\lambda_o(x:y^n)=1+\frac{1}{n},
\qquad
\lambda_o(x^2:y^2)=1,
$$
and
$$
\lambda_o(x^2:y^3)=\frac{5}{6},
\qquad
\lambda_o(x^n:y^n)=\frac{2}{n}.
$$
\end{exa}

\medskip

\begin{exa}[Map from cone over a rational curve of degree $n$]
Let $n\geq 2$ and let
$$
\sigma=\langle u_1,u_2\rangle
=
\langle(1,0),(-1,n)\rangle
\subset N_{\mathbb R}.
$$
Let $X=X_\sigma$ be the corresponding affine toric surface, and let
$x=x_\sigma$ be its torus-fixed point. Consider the toric rational map
$$
f\colon X\dashrightarrow\PP^1
$$
induced by the homomorphism
$$
\alpha\colon N=\Z^2\longrightarrow\Z,
\qquad
\alpha(a,b)=a.
$$
We have
$$
\alpha(u_1)=1,
\qquad
\alpha(u_2)=-1.
$$
Thus $x$ is an indeterminacy point of $f$. By the toric surface corollary,
$$
\lambda_x(f)
=
\frac{1}{|\alpha(u_1)|}
+
\frac{1}{|\alpha(u_2)|}
=
1+1
=
2.
$$

The kernel of $\alpha$ is generated by
$$
w=(0,1),
$$
which lies in the interior of $\sigma$. Hence the normalised graph is
obtained by subdividing $\sigma$ along the ray $\mathbb R_{\geq 0}w$.
The resulting cones
$$
\langle u_1,w\rangle
\qquad\text{and}\qquad
\langle w,u_2\rangle
$$
are smooth, since
$$
\det(u_1,w)=1,
\qquad
\det(w,u_2)=1.
$$
Thus the normalised graph is the minimal resolution of $X$.
\end{exa}

%%%%%%%%%%%%%%%%%%%%%%%%%%%%%%
\subsection{Non-toric computations}

We work out some non-toric examples.

\begin{exa}
First consider
$$
f\colon \mathbb A^2\dashrightarrow\mathbb P^1,
\qquad f=(x^7:g_5), \qquad
g_5=y(y-x)(y-2x)(y-3x)(y-4x)
$$
which is singular at the origin $o$.
The base ideal is
$$
I_f=(x^7,g_5).
$$
Blow up the origin, and let $E_0$ be the exceptional divisor. Since
$$
\operatorname{ord}_o(x^7)=7,
\qquad
\operatorname{ord}_o(g_5)=5,
$$
we have
$$
\operatorname{ord}_{E_0}(I_f)=5,
\qquad
a(E_0,\A^2,0)=2.
$$

On the chart $x=u$, $y=uv$, pullback of $x^7$ gives $u^7$ and pullback of $g_5$ gives
$$
u^5v(v-1)(v-2)(v-3)(v-4),
$$
so the pullback of $I_f$ is
$$
u^5\bigl(u^2,v(v-1)(v-2)(v-3)(v-4)\bigr).
$$
Thus there are five remaining base points on $E_0$, corresponding to
$$
v=0,1,2,3,4.
$$
Near any one of these points, with local parameter $t=v-a$, the ideal is,
up to multiplication by a unit,
$$
u^5(u^2,t).
$$

Blowing up this point produces an exceptional divisor $E_{a,1}$ satisfying
$$
\operatorname{ord}_{E_{a,1}}(I_f)=6,
\qquad
a(E_{a,1},\A^2,0)=3.
$$
One further blow-up principalises the ideal and produces an exceptional
divisor $E_{a,2}$ satisfying
$$
\operatorname{ord}_{E_{a,2}}(I_f)=7,
\qquad
a(E_{a,2},\A^2,0)=4.
$$
Consequently, the log discrepancy-to-order ratios appearing on this log
resolution are
$$
\frac{2}{5},
\qquad
\frac{3}{6}=\frac{1}{2},
\qquad
\frac{4}{7}.
$$
Their minimum is $2/5$. Therefore
$$
\lambda_0(f)=\frac{2}{5}.
$$
\end{exa}

\begin{exa}
Similarly, consider
$$
h\colon\mathbb A^2\dashrightarrow\mathbb P^1,
\qquad h=(x^n:y(y-x)), \qquad
n>2
$$
which is singular at the origin $o$. 
Its base ideal is
$$
I_h=(x^n,y(y-x)).
$$
The blow-up of the origin gives an exceptional divisor $E_0$ with
$$
\operatorname{ord}_{E_{0}}(I_h)=2,
\qquad
a(E_{0},\A^2,0)=2.
$$

On the chart $x=u$, $y=uv$, the pullback ideal is
$$
u^2(u^{n-2},v(v-1)).
$$
Hence there are two remaining base points on $E_0$, corresponding to
$v=0$ and $v=1$. Near either point, with local parameter $t$, the ideal
is, up to multiplication by a unit,
$$
u^2(u^{n-2},t).
$$

Principalising this ideal requires a sequence of $n-2$ point blow-ups.
If $E_{a,k}$ is the exceptional divisor produced by the $k$th blow-up in
one of these two sequences, where $1\leq k\leq n-2$, then
$$
\operatorname{ord}_{E_{a,k}}(I_h)=k+2,
\qquad
a(E_{a,k},\A^2,0)=k+2.
$$
Thus every exceptional divisor appearing after the first blow-up has log
discrepancy-to-order ratio
$$
\frac{a(E_{a,k},\A^2,0)}
{\operatorname{ord}_{E_{a,k}}(I_h)}
=1.
$$
The divisor $E_0$ also gives the ratio $2/2=1$. Therefore
$$
\lambda_0(h)=1.
$$
\end{exa}

%%%%%%%%%%%%%%%%%%%%%%%%%%%%%%
\subsection{Decomposition of maps via thresholds}

Let
$$
f\colon X\dashrightarrow Z
$$
be a birational map over a base $U$, where $X$ and $Z$ are projective over
$U$ and $X$ is $\Q$-factorial and klt. Choose a resolution
$$
p\colon X'\longrightarrow X,
\qquad
q\colon X'\longrightarrow Z
$$
of $f$, and let
$$
M':=q^*M_Z,
$$
where $M_Z$ is a divisor on $Z$ which is ample over $U$. The divisor $M'$
defines a nef b-divisor $M$ over $X$.

Fix $\epsilon\geq 0$ and assume that $(X,0)$ is $\epsilon$-lc. Define the
generalised $\epsilon$-lc threshold
$$
\lambda^\epsilon(M,X)
:=
\sup\{t\geq 0\mid (X,tM)\text{ is generalised $\epsilon$-lc}\}.
$$
The cases $\epsilon=0$ and $\epsilon=1$ give the generalised lc and
canonical thresholds, respectively.

If the threshold is finite, one takes a suitable crepant model extracting
certain divisors which compute it. Repeating the construction produces models
whose generalised $\epsilon$-lc thresholds form a strictly increasing
sequence
$$
\lambda_0<\lambda_1<\lambda_2<\cdots.
$$
Assuming the ACC for generalised $\epsilon$-lc thresholds, this process
terminates. At the final step the threshold is $+\infty$, which means that
the nef part descends to the resulting model. Since $M_Z$ is 
ample, the induced rational map to $Z$ is then a morphism. Thus the
threshold process resolves $f$.

The required ACC for generalised $\epsilon$-lc thresholds is conjectural in general, except in dimension $2$.
But for $\epsilon=0$ it is known in any dimension \cite{Birkar-Zhang, HMX-ACC}. A similar increasing-threshold termination argument is used
in \cite[3.10]{Birkar-boundedness-volume-generalised-pairs}.

\begin{rem}[The Sarkisov program and Cremona groups]
Let
$$
f\colon X\dashrightarrow Z
$$
be a birational map between Mori fibre spaces
$$
X\longrightarrow S
\qquad\text{and}\qquad
Z\longrightarrow T
$$
over a base $U$. The Sarkisov program decomposes $f$ into a finite sequence
of Sarkisov links. The threshold construction above is analogous to the Noether--Fano method, which, in modern language, uses generalised canonical thresholds in the study of birational maps of smooth surfaces and terminal threefolds (cf. \cite{Corti}).

A classical example is the quadratic Cremona transformation
$$
f\colon X=\PP^2\dashrightarrow Z=\PP^2,
\qquad
[x_0:x_1:x_2]
\longmapsto
[x_1x_2:x_0x_2:x_0x_1].
$$
It is resolved by blowing up the three coordinate points and then
contracting the strict transforms of the three coordinate lines. These
strict transforms are pairwise disjoint $(-1)$-curves.

Let $M_Z$ be a general hyperplane on $Z$. Then $M=f^*M_Z$ is a conic
passing through the three coordinate points. If
$$
p\colon Y\longrightarrow X
$$
is the blowup of these points, with exceptional divisors
$E_1,E_2,E_3$, and 
$$
q\colon Y\longrightarrow Z
$$ 
the induced morphism, then
$$
p^*M=q^*M_Z+E_1+E_2+E_3
$$
and
$$
K_Y=p^*K_X+E_1+E_2+E_3.
$$
Thus
$$
K_Y+q^*M_Z=p^*(K_X+M).
$$
It follows that the generalised canonical threshold is
$$
\lambda^1(M,X)=1,
$$
and its canonical centres are the three coordinate points, equivalently
the pairwise intersections of the coordinate lines. The blowup
$p\colon Y\to X$ is the corresponding crepant model. On $Y$ the map to
$Z$ is a morphism, so the nef part descends and the resulting threshold is
$+\infty$.
\end{rem}

%%%%%%%%%%%%%%%%%%%%%%%%%%%%%%
%%%%%%%%%%%%%%%%%%%%%%%%%%%%%%
\section{\bf Log discrepancies of normalised graphs}
\label{s-log-disc-normalised-graphs}

Given a rational map
$$
f\colon X\dashrightarrow \PP^n,
$$
one way to measure its singularities is by measuring the singularities of the normalised graph $Y^\nu$. However, when $X$ is singular, the singularities of $Y^\nu$ may reflect the singularities of $X$ rather than those of the map. Thus this approach is most useful when $X$ has very mild singularities, and especially when $X$ is smooth.

Even when $X$ is smooth and $f$ has mild singularities as measured by the normalised graph fibre degree and the generalised lc threshold, the normalised graph $Y^\nu$ can have deep singularities. Thus the singularities of $Y^\nu$ can detect phenomena which are invisible to the preceding invariants. In this sense, they give a finer and more sensitive measure of the singularities of the map.

In general, when $K_{Y^\nu}$ is $\Q$-Cartier we define the \emph{normalised graph mld} invariant 
$$
\theta_x(f)=\inf\{a(D,Y^\nu,0) \mid \mbox{$D$ prime divisor over $Y^\nu$ with centre $C_XD$ containing $x$}\}. 
$$
In particular,
$$
\theta_x(f)\ge 0
$$
if and only if $Y^\nu$ is lc over a neighbourhood of $x$. When this holds, we may also consider stronger relative properties, such as whether $Y^\nu$ is of Fano type over a neighbourhood of $x$. 

\begin{rem}
For surfaces we can avoid the $\mathbb Q$-Cartier assumption on $K_{Y^\nu}$ by using numerical log discrepancies. Let $\pi\colon W\to V$ be the minimal resolution of a normal surface, with exceptional curves $E_i$, and let $\Delta=\sum d_iE_i$ be the unique exceptional $\mathbb Q$-divisor determined by
$$
(K_W+\Delta)\cdot E_i=0
$$
for every $i$. We say that $V$ is \emph{numerically klt}, respectively \emph{numerically lc}, if $d_i<1$, respectively $d_i\leq1$, for every $i$; see \cite[Notation 4.1]{Kollar-Mori}. These notions make sense without assuming that $K_V$ is $\mathbb Q$-Cartier. Moreover, for normal surfaces numerical klt and numerical lc singularities are klt and lc, respectively; in particular the required $\mathbb Q$-Cartier property is then automatic. Thus in the surface examples below we may detect klt, lc, or non-lc singularities directly from the divisor $\Delta$.
\end{rem}

In the rest of this section, we give several examples illustrating the behaviour of this invariant.  
See \cite{Ambro-toric-mld} for discussions on toric log discrepancies that will be used below. 

%%%%%%%%%%%%%%%%%%% 
\subsection{A toric map of degree one with $X,Y^\nu$ having deep singularities} 

\begin{exa}\label{exa-toric-map-degree-one-X-Y-deep-sing}
Let $n\ge 2$, let $N=\mathbb Z^2$, and consider the cone
$$
\sigma=\langle u_1,u_2\rangle\subset N_{\mathbb R},
$$
where
$$
u_1=(-1,n)
\qquad\mbox{and}\qquad
u_2=(1,0).
$$
Let
$
X=X_\sigma
$
be the corresponding affine toric surface, and denote its closed torus-fixed point by $x=x_\sigma$.

Consider the toric rational map
$$
f\colon X\dashrightarrow\PP^1
$$
induced by the homomorphism
$$
\alpha\colon N\longrightarrow\mathbb Z,
\qquad
\alpha(a,b)=a-b.
$$
We have
$$
\alpha(u_1)=-(n+1)
\ 
\mbox{and} \
\alpha(u_2)=1.
$$
Thus $\alpha$ takes opposite signs on the two rays of $\sigma$, so $f$ is singular at $x$.

The kernel of $\alpha$ is generated by the primitive vector
$$
w=(1,1).
$$
Since $w$ is in the interior of $\sigma$, the normalised graph
$$
p\colon Y^\nu\longrightarrow X
$$
is the toric modification obtained by subdividing $\sigma$ along the ray $\mathbb R_{\ge 0}w$. Its two maximal cones are
$$
\sigma_1=\langle u_1,w\rangle
\ 
\mbox{and} \ 
\sigma_2=\langle w,u_2\rangle.
$$
As we have seen before, the minimal log discrepancy of $X$ at $x$ is 
$
\frac{2}{n}.
$

We next calculate the singularities of $Y^\nu$ over $x$. The cone $\sigma_2$ is smooth because
$$
\left|\det(w,u_2)\right|=1.
$$
Thus the only singular torus-fixed point of $Y^\nu$ over $x$ corresponds to the cone $\sigma_1$.
The lattice point $(0,1)$ belongs to the interior of $\sigma_1$ and its corresponding divisor $D_{(0,1)}$ satisfies
$$
a(D_{(0,1)},Y^\nu,0)=\frac{2}{n+1}
$$
as 
$$
(0,1)=\frac{1}{n+1}u_1+\frac{1}{n+1}w.
$$

Now let $(a,b)\in N$ be in the interior of $\sigma_1$. Then  
$$
(a,b)=su_1+tw
$$
where we can calculate
$$
s=\frac{b-a}{n+1}>0, \qquad t=a+s>0.
$$
Therefore
$$
a(D_{(a,b)},Y^\nu,0)=s+t\ge\frac{2}{n+1}.
$$
It follows that  
$$
\theta_x(f)=\frac{2}{n+1},
$$
that is, $Y^\nu$ is exactly $\frac{2}{n+1}$-lc over $x$.

We calculate the normalised graph fibre degree using Theorem~\ref{thm:toric-normalised-graph-fibre-degree}. 
We first calculate $\mu_\sigma(w)$ as in the theorem. Let $M$ be the dual lattice and consider
$$
m=(0,1)\in M.
$$
Then
$$
\langle m,u_1\rangle=n
\qquad\mbox{and}\qquad
\langle m,u_2\rangle=0,
$$
so
$$
m\in\sigma^\vee\cap M.
$$
Moreover,
$
\langle m,w\rangle=1.
$
Therefore
$
\mu_\sigma(w)=1.
$
More precisely, if $E=D_w$ is the exceptional divisor of the normalised graph
$$
p\colon Y^\nu\longrightarrow X,
$$
then the fundamental cycle 
$
[Y^\nu_x]=E,
$
and the restriction
$
q|_E\colon E\longrightarrow\PP^1
$
has degree one. Thus
$$
E\simeq\PP^1
\qquad\mbox{and}\qquad
\mathcal{O}_{Y^\nu}(1)|_E
\simeq
\mathcal{O}_{\PP^1}(1).
$$
In particular, $\delta_x(f)=1$. In fact, this implies $f$ is linear type at $x$, by Theorem \ref{t-degree-one-linear-type-rational-surfaces} below.

Finally we calculate $\lambda_x(f)$. We have
$$
|\alpha(u_1)|=n+1
\
\mbox{and} \
|\alpha(u_2)|=1.
$$
Therefore, by Corollary \ref{cor-rational-map-lc-threshold-toric-surface}, 
$$
\lambda_x(f)=1+\frac{1}{n+1}=\frac{n+2}{n+1}.
$$

We have therefore shown that
$$
\mld_x(X)=\frac{2}{n},
\qquad
\theta_x(f)=\frac{2}{n+1},
\qquad
\delta_x(f)=1,
\qquad
\lambda_x(f)=\frac{n+2}{n+1}.
$$
As $n$ tends to infinity, both $X$ and $Y^\nu$ have arbitrarily deep klt singularities, although $f$ remains to have the simplest kind of map singularities, that is, linear type.
\end{exa}

%%%%%%%%%%%%%%%%%%%
\subsection{A toric map of degree one with canonical $Y$}

\begin{exa}
\label{exa-rational-map-(x : y^n)}
Consider the rational map
$$
f\colon\A^2\bir\PP^1,
\qquad
f=(x:y^n),
\qquad
n\ge 1.
$$
It is {singular} only at the origin $o=(0,0)$. If $(u:v)$ are homogeneous coordinates on $\PP^1$, then the graph
$$
Y\subset\A^2\times\PP^1
$$
is defined by
$$
uy^n=vx.
$$
The projection
$$
\pi\colon Y\longrightarrow\A^2
$$
is the blowup of the ideal
$$
I=(x,y^n)\subset k[x,y].
$$

The surface $Y$ is normal. Indeed, on the chart $v=1$, it is given by
$$
x=uy^n,
$$
so this chart is smooth. On the chart $u=1$, it is given by
$$
y^n=vx,
$$
which has an $A_{n-1}$ singularity at the origin. In particular, $Y$ has canonical singularities and
$$
Y^\nu=Y.
$$
Torically, $\pi$ is the weighted blowup of $\A^2$ with weights $(n,1)$ where the $x=0$ line corresponds to the vector $(1,0)$ in the standard cone defining $\A^2$.

The scheme-theoretic fibre over $o$ is
$$
Y_o={o}\times\PP^1\simeq\PP^1.
$$
Moreover, the morphism
$$
Y_o\longrightarrow\PP^1
$$
induced by the second projection is an isomorphism. Therefore
$$
\mathcal{O}_Y(1)|_{Y_o}\simeq\mathcal{O}_{\PP^1}(1),
$$
so $f$ is linear type at $o$.

Let $E$ be the exceptional Cartier divisor determined by
$$
I\mathcal{O}_Y=\mathcal{O}_Y(1)=\mathcal{O}_Y(-E),
$$
and let $C=Y_o$ be the reduced exceptional curve. On the chart $v=1$, we have
$$
x=uy^n,
$$
and hence
$$
I\mathcal{O}_Y=(uy^n,y^n)=(y^n).
$$
Since $C$ is defined generically by $y=0$, it follows that
$$
E=nC.
$$
Thus
$$
I\mathcal{O}_Y=
\mathcal{O}_Y(1)=
\mathcal{O}_Y(-nC).
$$

Resolving $f$ requires $n$ successive point blowups. The exceptional locus is a chain
$$
E_1\mathbin{-}E_2\mathbin{-}\cdots\mathbin{-}E_{n-1}\mathbin{-}E_n,
$$
where
$$
E_i^2=-2
$$
for $1\le i\le n-1$, while
$$
E_n^2=-1.
$$
The induced morphism to $Y$ contracts $E_1,\ldots,E_{n-1}$ and maps $E_n$ isomorphically onto $C$. The contracted chain is the minimal resolution of the $A_{n-1}$ singularity of $Y$.
\end{exa}

%%%%%%%%%%%%%%%%%%% 
\subsection{A map of degree two on $\A^2$ with $Y^\nu$ having deep klt singularities}

\begin{exa}\label{exa-map-degree-two-on-A2-Y-deep-klt}
Consider the rational map
$$
f\colon\A^2\bir\PP^1,
\qquad
f=(x^n:y(y-x)),
\qquad
n>2.
$$
This is {singular} only at the origin $o=(0,0)$. Let
$$
\phi_1\colon X_1\longrightarrow X=\A^2
$$
be the blowup of $o$, with exceptional divisor $E_1$, and let
$$
f_1\colon X_1\bir\PP^1
$$
be the induced map. In the chart $x=u$, $y=uv$, we have
$$
f_1=(u^n:u^2v(v-1))=(u^{n-2}:v(v-1)).
$$
Thus $f_1$ is undefined at the two points
$$
P_1=(0,0)
\qquad\mbox{and}\qquad
P_2=(0,1)
$$
on $E_1$. In the other chart $x=wz$, $y=z$, the induced map is
$$
f_1=(w^nz^{n-2}:1-w),
$$
which is undefined at $(1,0)$, corresponding to $P_2$.

Near each $P_i$, there are local coordinates $\alpha,\beta$ in which $f_1$ is given by
$$
(\alpha^{n-2}:\beta).
$$
Hence each $P_i$ is resolved by a sequence of $n-2$ blowups, as in Example~\ref{exa-rational-map-(x : y^n)}.

Let
$$
W\longrightarrow X
$$
be the resulting resolution. Its exceptional locus consists of a central $(-3)$-curve, with two identical chains attached to it. Each chain consists of $n-3$ curves of self-intersection $-2$ followed by a terminal $(-1)$-curve. When $n=3$, each chain consists only of the terminal $(-1)$-curve.
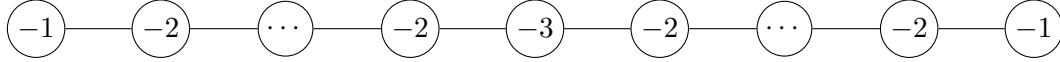
\begin{center}
\begin{tikzpicture}[
every node/.style={circle,draw,inner sep=2pt,minimum size=20pt},
scale=1.1
]
\node (Llast) at (0,0) {$-1$};
\node (Lm) at (1.5,0) {$-2$};
\node (Ldots) at (3,0) {$\cdots$};
\node (L1) at (4.5,0) {$-2$};
\node (E1) at (6,0) {$-3$};
\node (R1) at (7.5,0) {$-2$};
\node (Rdots) at (9,0) {$\cdots$};
\node (Rm) at (10.5,0) {$-2$};
\node (Rlast) at (12,0) {$-1$};

\draw
(Llast) -- (Lm) -- (Ldots) -- (L1) -- (E1)
-- (R1) -- (Rdots) -- (Rm) -- (Rlast);
\end{tikzpicture}
\end{center}

Let
$$
h\colon W\longrightarrow Y^\nu
$$
be the induced morphism to the normalised graph. The morphism $W\to\PP^1$ is constant on the central $(-3)$-curve and on all the $(-2)$-curves, but is nonconstant on the two terminal $(-1)$-curves. Therefore $h$ contracts precisely the central $(-3)$-curve and the two attached chains of $(-2)$-curves. 

Denote the central $(-3)$-curve by $C_0$, and denote the $(-2)$-curves at distance $i$ from $C_0$ by
$$
C_i^+
\qquad\mbox{and}\qquad
C_i^-,
\qquad
1\le i\le n-3.
$$
Write
$$
K_W=
h^*K_{Y^\nu}
+
a_0C_0
+
\sum_{i=1}^{n-3}a_i(C_i^++C_i^-).
$$
Intersecting with the contracted curves gives
$$
a_0=-\frac{n-2}{n}, \qquad
a_i=-\frac{n-2-i}{n}.
$$
Thus
$$
a(C_0,Y^\nu,0)=\frac{2}{n},
$$
while
$$
a(C_i^\pm,Y^\nu,0)=\frac{i+2}{n}.
$$
Therefore
$$
\theta_o(f)=\mld(Y^\nu)\ \mbox{over }o=\frac{2}{n}.
$$

On the other hand, for a general hyperplane $H\subset\PP^1$, the divisor $f^*H$ is locally defined at $o$ by
$$
ax^n+by(y-x)=0,
$$
where $b\neq 0$. Hence
$$
\mu_of^*H=2.
$$
By Theorem~\ref{t-norm-graph-fibre-degree-mult-hyperplane-pullback},
$$
\delta_o(f)=2.
$$
The same resolution shows that every log discrepancy-to-order ratio for the base ideal
$$
I_f=(x^n,y(y-x))
$$
is equal to $1$. Hence
$$
\lambda_o(f)=1.
$$

Thus
$$
\delta_o(f)=2,
\qquad
\lambda_o(f)=1,
\qquad
\theta_o(f)=\frac{2}{n}.
$$
In particular, the normalised graph fibre degree and the generalised lc threshold remain fixed, while the singularities of $Y^\nu$ become arbitrarily deep as $n$ increases.
\end{exa}

%%%%%%%%%%%%%%%%%%% 
\subsection{A toric map of degree two on $\A^2$ with $Y^\nu$ having deep singularities}

\begin{exa}\label{exa-toric-map-degree-two-A2-Y-deep-klt}
Let $n\ge 2$ and consider the rational map
$$
f\colon\A^2\bir\PP^2,
\qquad
f=(xy:x^{n+1}:y^{n+1}),
$$
which is {singular} only at the origin $o=(0,0)$.

Let $H\subset\PP^2$ be a general hyperplane. Then $f^*H$ is defined near $x$ by
$$
axy+bx^{n+1}+cy^{n+1}=0,
$$
where $a\neq 0$. Hence
$$
\mu_of^*H=2.
$$
By Theorem \ref{t-norm-graph-fibre-degree-mult-hyperplane-pullback}, the normalised graph fibre degree 
$$
\delta_x(f)=2.
$$

We next calculate the generalised lc threshold. The map $f$ is induced by the homomorphism
$$
\alpha\colon\mathbb Z^2\longrightarrow\mathbb Z^2,
\qquad
\alpha(a,b)=(na-b,-a+nb).
$$
The fan of the normalised graph $Y^\nu$ is obtained from the first quadrant by inserting the rays generated by
$$
w_1=(n,1)
\qquad\mbox{and}\qquad
w_2=(1,n)
$$
which are mapped by $\alpha$ to the rays generated by $(1,0)$ and $(0,1)$, respectively. 

This time take $H\subset \PP^2$ to be the divisor corresponding to the vector $(-1,-1)$ in the fan of $\PP^2$.  
We need the coefficients of the exceptional fixed part 
$$
E_f=p^*f^*H-q^*H. 
$$ 
Using the support function of $H$ we can see that 
$$q
^*H=D_{(1,0)}+D_{(0,1)}.
$$ 
Next using the support function of $f^*H=p_*q^*H$ we can see that 
$$
E_f=(n+1)D_{w_1}+(n+1)D_{w_2}.
$$

Since the toric log discrepancy function of $\A^2$ is
$$
\psi(a,b)=a+b,
$$
we have
$$
\psi(w_1)=\psi(w_2)=n+1.
$$
The toric threshold formula of Proposition \ref{p-rational-map-lc-threshold-toric-case} therefore gives
$$
\lambda_o(f)=1.
$$

It remains to calculate the singularities of $Y^\nu$. Its maximal cones are
$$
\langle(1,0),w_1\rangle,
\qquad
\sigma_n:=\langle w_1,w_2\rangle,
\qquad
\langle w_2,(0,1)\rangle.
$$
The first and third cones are smooth, so the only singular point of $Y^\nu$ corresponds to $\sigma_n$.

The toric log discrepancy function on $\sigma_n$ is the linear function $\ell_n$ satisfying
$$
\ell_n(w_1)=\ell_n(w_2)=1.
$$
Thus
$$
\ell_n(a,b)=\frac{a+b}{n+1}.
$$
The lattice point $(1,1)$ lies in $\operatorname{Int}(\sigma_n)$ and
$$
\ell_n(1,1)=\frac{2}{n+1}.
$$
On the other hand, every lattice point $(a,b)\in\operatorname{Int}(\sigma_n)$ satisfies $a,b>0$, and hence $a+b\ge 2$. Therefore
$$
\theta_o(f)=\mld(Y^\nu)\ \mbox{over }o=\frac{2}{n+1}.
$$
Thus
$$
\delta_o(f)=2,
\qquad
\lambda_o(f)=1,
\qquad
\theta_o(f)=\frac{2}{n+1}.
$$
In particular, the normalised graph fibre degree and the generalised lc threshold remain fixed, while the singularities of $Y^\nu$ become arbitrarily deep klt as $n$ tends to infinity.
\end{exa}

%%%%%%%%%%%%%%%%%%% 
\subsection{A map of degree 5 on $\A^2$ with non-lc $Y^\nu$}

\begin{exa}\label{exa-map-degree-5-A2-non-lc-Y}
We give an example whose normalised graph is not lc. Consider
$$
f\colon \A^2\bir\PP^1,
\qquad
f=\bigl(x^7:y(y-x)(y-2x)(y-3x)(y-4x)\bigr),
$$
which is {singular} only at the origin $o$. For a general point $H\in\PP^1$, the curve $f^*H$ has multiplicity $5$ at $o$. Thus, by Theorem \ref{t-norm-graph-fibre-degree-mult-hyperplane-pullback},
$$
\delta_o(f)=5.
$$

Blow up $o$, and denote the exceptional curve by $E_1$. On the chart
$$
x=u,\qquad y=uv,
$$
the induced map is
$$
\bigl(u^2:v(v-1)(v-2)(v-3)(v-4)\bigr).
$$
Hence it has five singular points $P_1,\dots,P_5$ on $E_1$, corresponding to
$$
v=0,1,2,3,4.
$$
Near each of them, the map has the local form
$$
(\alpha^2:\beta).
$$

On the chart $x=wz$ and $y=z$, the induced map is 
$$
\bigl(w^7z^2 : (1-w)(1-2w)(1-3w)(1-4w)\bigr).
$$
The singular points are given by $z=0$ and $w=1,\frac12,\frac13,\frac14$ corresponding to $P_2,\dots,P_5$.

Resolving each point $P_i$ requires two blowups: the first produces a $(-2)$-curve and the second a $(-1)$-curve. Thus the exceptional locus of the resulting resolution
$$
W\longrightarrow\A^2
$$
consists of a central curve $E_1$ with
$$
E_1^2=-6,
$$
together with five arms of the form
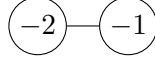
\begin{center}
\begin{tikzpicture}[baseline=-0.5ex]
\node[circle,draw,inner sep=2pt,minimum size=18pt] (E) at (0,0) {$-2$};
\node[circle,draw,inner sep=2pt,minimum size=18pt] (F) at (1.2,0) {$-1$};
\draw (E)--(F);
\end{tikzpicture}
\end{center}

Let
$$
h\colon W\longrightarrow Y^\nu
$$
be the induced morphism to the normalised graph. The morphism $h$ contracts the central $(-6)$-curve and the five $(-2)$-curves, while each final $(-1)$-curve maps nontrivially to $\PP^1$ and is not contracted.

Write the contracted $(-2)$-curves as $E_2,\ldots,E_6$. Consider 
$$
K_W+e_1E_1+\sum_{i=2}^6e_iE_i\equiv 0/Y^\nu.
$$
By symmetry, $e_2=\cdots=e_6=:e$. Intersecting with $E_1$ and $E_2$ gives
$$
4-6e_1+5e=0,
\ \
e_1-2e=0.
$$
Therefore
$$
e_1=\frac{8}{7}
\qquad\mbox{and}\qquad
e_2=\cdots=e_6=\frac{4}{7}.
$$
Since the coefficient of $E_1$ is greater than one, $Y^\nu$ is not lc.

For comparison, the normalised graph of
$$
\bigl(x^6:y(y-x)(y-2x)(y-3x)\bigr)
$$
is lc but not klt. Indeed, the analogous contracted configuration consists of a central $(-5)$-curve and four $(-2)$-curves, and the corresponding coefficients are
$$
e_1=1
\qquad\mbox{and}\qquad
e_2=\cdots=e_4=\frac12.
$$
\end{exa}

%%%%%%%%%%%%%%%%%%%%%%%%%%%%%%
%%%%%%%%%%%%%%%%%%%%%%%%%%%%%%
\section{\bf The generalised lc threshold--degree inequality on klt surfaces}

The invariants $\delta_x(f)$, $\lambda_x(f)$, and $\theta_x(f)$ introduced above measure different aspects of the singularity of a rational map. The examples in Section \ref{s-log-disc-normalised-graphs} show that $\theta_x(f)$ can vary substantially even when $\delta_x(f)$ and $\lambda_x(f)$ remain fixed. In this section we establish a direct relation between the normalised graph fibre degree $\delta_x(f)$ and the generalised lc threshold $\lambda_x(f)$.

\subsection{The inequality}

\begin{thm}[Generalised lc threshold--degree inequality]
\label{t-lct-degree-inequality-for-rational-maps-on-surfaces}
Let $x\in X$ be a closed point on a klt surface, and let
$$
f\colon X\dashrightarrow \mathbb P^n
$$
be a rational map which is singular at $x$. Then
$$
\delta_x(f)\lambda_x(f)\leq 2.
$$
\end{thm}
\begin{proof}
Take a resolution
$$
\phi\colon V\longrightarrow X
$$
such that the induced map
$$
h\colon V\bir\mathbb P^n
$$
is a morphism and
$$
m_x\mathcal O_V=\mathcal O_V(-N)
$$
for an effective exceptional divisor $N$. Let $H\subset\mathbb P^n$ be a sufficiently general hyperplane. Write
$$
\phi^*f^*H=h^*H+E_f
$$
and define $B_V$ by
$$
K_V+B_V=\phi^*K_X.
$$

Ffor eacse of notation let $0t=\lambda_x(f)$. Then the generalised pair $(X,t f^*H)$ is generalised lc at $x$, so every coefficient of
$$
\Delta_t:=B_V+tE_f
$$
along a divisor over $x$ is at most one. Moreover,
$$
K_V+\Delta_t+th^*H=\phi^*(K_X+tf^*H).
$$
Intersecting with $N$ gives
$$
(K_V+\Delta_t)\cdot N+th^*H\cdot N=0.
$$

We claim that
$$
(K_V+\Delta_t)\cdot N\geq-2.
$$
Indeed, since $m_x\mathcal O_V=\mathcal O_V(-N)$, the divisor $-N$ is nef over $X$. Moreover, as $X$ is rational, \cite[Theorem 12.1(ii)]{Lipman-rational-singularities} gives 
$$
R^1\phi_*\mathcal O_V(-N)=0.
$$
Using
$$
\phi_*\mathcal O_V(-N)= m_x
$$
and using direct image of the exact sequence
$$
0\longrightarrow\mathcal O_V(-N)
\longrightarrow\mathcal O_V
\longrightarrow\mathcal O_N
\longrightarrow0,
$$
we obtain
$$
H^0(N,\mathcal O_N)=k,
\qquad
H^1(N,\mathcal O_N)=0.
$$
Thus $\mathcal{X}(N,\mathcal{O}_N)=1$, so 
$$
p_a(N)=1-\mathcal{X}(N,\mathcal{O}_N)=0,
$$ 
and hence by adjunction
$$
(K_V+N)\cdot N=2p_a(N)-2=-2.
$$

Every exceptional prime divisor over $x$ appears in $N$ with coefficient at least one. Since the coefficients of $\Delta_t$ are at most one,
$$
N-\Delta_t\geq0
$$
along the support of $N$. As $N$ is anti-nef,
$$
(N-\Delta_t)\cdot N\leq0.
$$
Therefore
$$
(K_V+\Delta_t)\cdot N=(K_V+N)\cdot N+(\Delta_t-N)\cdot N\geq-2.
$$

It follows that
$$
th^*H\cdot N\leq2.
$$
Since $N$ is the scheme-theoretic fibre of $V\to X$ over $x$, and $V\to X\times\mathbb P^n$ factors through the normalised graph, the projection formula gives
$$
h^*H\cdot N=q^*H\cdot [Y^\nu_x]=\delta_x(f).
$$
Hence
$$
t\delta_x(f)\leq2
$$
which means
$$
\delta_x(f)\lambda_x(f)\leq2.
$$
\end{proof}

%%%%%%%%%%%%%%%%%%%%%%%%%%%%%%
\subsection{Bounded fibre degree but arbitrarily small threshold}

The constant $2$ in Theorem \ref{t-lct-degree-inequality-for-rational-maps-on-surfaces} is sharp. Indeed, 
$$
f=(x^m:y^m)\colon\mathbb A^2\dashrightarrow\mathbb P^1
$$
is singular at the origin $o$, and 
$$
\delta_o(f)=m
\qquad\text{and}\qquad
\lambda_o(f)=\frac{2}{m}
$$
where the first equality follows from Theorem \ref{t-norm-graph-fibre-degree-mult-hyperplane-pullback} and the second equality follows from Proposition \ref{p-rational-map-lc-threshold-toric-case} as $f$ is induced by $\alpha\colon \Z^2\to \Z$ sending $(a,b)$ to $ma-mb$.
Thus
$$
\delta_o(f)\lambda_o(f)=2.
$$

\begin{exa}\label{exa-bnd-fib-degree-small-threshold}
Theorem \ref{t-lct-degree-inequality-for-rational-maps-on-surfaces} is useful in the sense that a positive lower bound for the generalised lc threshold gives an upper bound on the normalised graph fibre degree. But the converse is not true, that is, the normalised graph fibre degree can be bounded from above while the generalised lc threshold can be arbitrarily small.

Indeed, let
$$
X=\{xy=z^{2r}\}\subset\mathbb A^3
$$
be the $A_{2r-1}$ singularity at the origin $o$. Let
$$
\phi\colon W\longrightarrow X
$$
be its minimal resolution. The exceptional locus is a chain
$$
E_1-\cdots-E_{2r-1},
\qquad
E_i^2=-2.
$$
Set
$$
F:=\sum_{i=1}^{2r-1}\min\{i,2r-i\}E_i.
$$
Then
$$
F\cdot E_i=0
\quad\text{for }i\neq r,
$$
whereas
$$
F\cdot E_r=-2.
$$
Then $-F$ is nef over $X$, actually, its linear system is base point free over $X$ as $W\to X$ is toric.

Choose a finite-dimensional subspace of $H^0(W,-F)$ generating $\mathcal{O}_W(-F)$ over $X$, and let
$$
h\colon W\longrightarrow\mathbb P^n
$$
be the induced morphism. Together with $\phi$, it determines a rational map
$$
f\colon X\dashrightarrow\mathbb P^n.
$$
Since
$
-F\cdot E_r=2,
$
the map $f$ is singular at the origin $o$.

Let $H\subset\mathbb P^n$ be a general hyperplane. Then
$$
h^*H\sim-F/X,
\qquad
\phi^*f^*H=h^*H+F.
$$
Moreover, the fundamental cycle of $W/X$ is 
$$
N=E_1+\cdots+E_{2r-1}
$$
and $\mathcal{O}_W(-N)$ is generated relatively over $X$. This combined with $\phi_*\mathcal{O}_W(-N)=m_x$ shows that  $m_x\mathcal{O}_W=\mathcal{O}_W(-N)$. 
Then
$$
\delta_o(f)=h^*H\cdot N=-F\cdot N=2.
$$

We now calculate the generalised lc threshold. Since $X$ is canonical,
$$
K_W=\phi^*K_X.
$$
Moreover, $h^*H$ is the trace of the nef part on $W$, and for every $t\geq0$ we have
$$
K_W+tF+th^*H=\phi^*(K_X+tf^*H).
$$
Thus the generalised boundary on $W$ is $tF$. Since $\operatorname{Supp}F$ has simple normal crossings, the generalised pair $(X,t f^*H)$ is generalised lc at $o$ precisely when every coefficient of $tF$ is at most one. As the largest coefficient of $F$ is $r$, we obtain
$$
\lambda_o(f)=\frac1r.
$$

Hence
$$
\delta_o(f)=2,
\qquad
\lambda_o(f)=\frac1r.
$$
In particular, the normalised graph fibre degree remains fixed while the generalised lc threshold tends to zero.
\end{exa}

%%%%%%%%%%%%%%%%%%%%%%%%%%%%%%
%%%%%%%%%%%%%%%%%%%%%%%%%%%%%%
\section{\bf Maps on smooth surfaces}

In this section we study singularities of rational maps on smooth surfaces in greater depth.
Assume we are given a rational map 
$$
f\colon X\bir \PP^n
$$
and a closed point $x\in X$. Assume $X$ is $\epsilon$-lc at $x$ for $\epsilon\geq 0$ and that $f^*H$ is $\Q$-Cartier at $x$. Define the \emph{generalised $\epsilon$-lc threshold of $f$ at $x$} by
$$
\lambda_x^\epsilon(f):=\sup\{t\geq 0\mid (X,tf^*H)\text{ is generalised $\epsilon$-lc at }x\},
$$
where $H\subset \PP^n$ is a hyperplane and $(X,tf^*H)$ is viewed as a generalised pair as above with nef part the pullback of $H$ to the normalised graph. Thus
$$
\lambda_x^0(f)=\lambda_x(f).
$$

%%%%%%%%%%%%%%%%%%%%%%%%%%%%%%
\subsection{Multiplicity, fibre degree and the generalised $1$-lc threshold}

\begin{lem}[Multiplicity, fibre degree and the generalised $1$-lc threshold]\label{l-rational-maps-mult=fib-degree=inverse-can-threshold}
Let $f\colon X\bir \PP^n$ be a rational map from a smooth surface, let $x\in X$ be a closed point, and assume $f$ is singular at $x$. Let $H\subset \PP^n$ be a general hyperplane. Then
$$
\mu_xf^*H=\delta_x(f)=\frac{1}{\lambda_x^1(f)}.
$$
\end{lem}

\begin{proof} 
Let $Y^\nu$ be the normalised graph of $f$ and $Y^\nu_x$ be its scheme-theoretic fibre over $x$.
Taking the minimal resolution $W\to Y^\nu$ and denote the induced morphism $W\to \PP^1$ by $h$.
Then $K_W+h^*H$ is nef over $X$ because $K_W$ is nef over $Y^\nu$ and because  for any $-1$-curve $C$ we have $K_W\cdot C=-1$ and $h^*H\cdot C\ge 1$. Thus running the MMP on $K_W$ over $X$ with scaling of $h^*H$
decomposes $W\to X$ into a sequence of smooth blowups 
$$
W=X_l\to \cdots \to X_0=X.
$$ 
Denote the induced map $X_i\bir \PP^n$ by $f_i$. 

Let $E_i$ be the exceptional divisor of $\phi_i\colon X_i\to X_{i-1}$ which blows up $x_{i-1}$. If $\tau$ is the number in the MMP appearing in the last contraction $X_1\to X_0=X$, then 
$$
K_{X_1}+\tau f_1^*H=\phi_1^*(K_X+\tau f^*H).
$$ 
Moreover, since $(W,f_l^*H)$ has generalised canonical singularities with nef part $f_l^*H$, $(X,\tau f^*H)$ has generalised canonical singularities but $(X,tf^*H)$ is not generalised canonical for any $t>\tau$. In other words, 
$$
\lambda_x^1(f)=\tau.
$$ 
Moreover, $\tau f_1^*H\cdot E_1=1$, and since $E_1$ is not a component of $f_1^*H$, we get   
$$
\mu_xf^*H=f_1^*H\cdot E_1=\frac{1}{\lambda_x^1(f)}.
$$ 

The equality $\mu_xf^*H=\delta_x(f)$ is immediate by Theorem \ref{t-norm-graph-fibre-degree-mult-hyperplane-pullback}.
\end{proof}

%%%%%%%%%%%%%%%%%%%
\subsection{Bounded singularity negativity of normalised graphs}

Let $V$ be a normal surface. We say that $V$ is $l$-bounded if, on the minimal resolution
$$
\pi\colon W\to V,
$$
every exceptional curve $E$ satisfies
$$
E^2\geq-l,
$$
and there are at most $l$ exceptional curves $E$ with
$$
E^2\leq-3.
$$

Thus an $l$-bounded surface may have arbitrarily many $(-2)$-curves. Moreover, $V$ may not be lc even if $K_V$ is $\Q$-Cartier.

\begin{thm}\label{t-bnd-singularity-negativity-normalised-graphs}
Let $f\colon X\bir\PP^n$ be a rational map from a smooth surface, let $x\in X$ be a closed point, and assume that $f$ is singular only at $x$. Let
$$
p\colon Y^\nu\to X
$$
be the normalised graph. If
$$
\delta_x(f)\leq d,
$$
then $Y^\nu$ is $(3d+1)$-bounded.
\end{thm}

\begin{proof}
As in the proof of Lemma \ref{l-rational-maps-mult=fib-degree=inverse-can-threshold}, taking the minimal resolution $W\to Y^\nu$ and running the MMP with scaling 
decomposes $W\to X$ into a sequence of smooth blowups 
$$
W=X_l\to \cdots \to X_0=X.
$$ 
Denote the induced map $X_j\bir \PP^n$ by $f_j$. 

 We will abuse notation and denote by $E_i$ the birational transform of the exceptional curve of $X_i\to X_{i-1}$ on each $X_j$, $i\le j$. On $X_j$ let
$$
S_j=f_j^*H,\qquad Z_j=\sum_{i=1}^jE_i,
$$
where $H\subset\PP^n$ is general.

Pick $j\ge 2$. If $X_j\to X_{j-1}$ is a single blowup, that is, if it blows up a point on only one irreducible component of the exceptional locus of $X_{j-1}\to X$,  then on $X_j$ we have 
$$
S_j\cdot Z_j=S_{j-1}\cdot Z_{j-1}.
$$
But if $X_j\to X_{j-1}$ is a double blowup, that is, if it blows up a point on two irreducible components of the exceptional locus of $X_{j-1}\to X$, then  on $X_j$ we have
$$
S_j\cdot Z_j=S_{j-1}\cdot Z_{j-1}-S_j\cdot E_j.
$$
Moreover, on $X_1$ we have
$$
S_1\cdot E_1=\mu_xf^*H=\delta_x(f)\leq d.
$$
Additionally, $S_j\cdot E_j\geq1$ on $X_j$ because by the proof of Lemma \ref{l-rational-maps-mult=fib-degree=inverse-can-threshold}, $W\to X$ is an MMP with scaling of $f_l^*H$, so $S_j\cdot E_j>0$. Thus the number $r$ of double blowups in the sequence is at most $d$ as $S_l\cdot Z_l\ge 0$.

On $X_j$, let
$$
h(j):=\#(S_j\cap Z_j),\qquad s(j):=\#(S_j\cap E_j)
$$
where $\#$ denotes the number of points, set-theoretically.
Then  for $j\ge 1$, we have 
$$
h(j)=h(j-1)+s(j)-1,
$$
where by convention we put $h(0)=1$. Therefore,
$$
\sum_{j=1}^l(s(j)-1)=h(l)-1\leq d-1
$$
because 
$$
h(l)\le S_l\cdot Z_l\le S_1\cdot E_1\le d.
$$

For each $j$, let $a_j$ and $b_j$ be the numbers of single and double blowups, respectively, in the sequence $W\to X_j$ whose centres lie on (the birational transform of) $E_j$. Then on $W$, 
$$
-E_j^2-1=a_j+b_j.
$$
Every later single blowup centred on the transform of $E_j$ occurs at a point of $S_j\cap E_j$, and two such single centres give distinct points on $E_j$. If one continues blowing up above the same point of $E_j$, the next centre on $E_j$ is a double centre, because it lies at the intersection with the newly created exceptional curve. Hence the single centres are bounded by the number of points in $S_j\cap E_j$. Thus $a_j\leq s(j)$.

Therefore, on $W$ we have 
$$
\max\{-E_j^2-2,0\}\leq s(j)-1+b_j.
$$
Since every double blowup contributes to two of the numbers $b_j$,
$$
\sum_{j=1}^lb_j=2r.
$$
Consequently,
$$
\sum_{j=1}^l\max\{-E_j^2-2,0\}\le \sum_{j=1}^l(s(j)-1+b_j)\leq d-1+2r\leq3d-1.
$$
Thus there are at most $3d-1$ curves with $E_j^2\leq-3$, and each such curve satisfies
$$
-E_j^2\leq3d+1.
$$
This shows that $Y^\nu$ is $(3d+1)$-bounded.
\end{proof}

%%%%%%%%%%%%%%%%%%%%%%%
\subsection{Degree spectrum and generalised $1$-lc thresholds in blowup sequence}

\begin{thm}\label{t-degree-spectrum-infinitely-near-thresholds}
Let $f\colon X\bir\PP^n$ be a rational map from a smooth surface and $x\in X$ a closed point. Assume $f$ has an isolated singularity at $x$. Consider a sequence of point blowups
$$
W=X_l\longrightarrow\cdots\longrightarrow X_1\longrightarrow X_0=X
$$
where $X_{i+1}\to X_i$ blows up $x_i$, and let $f_i\colon X_i\bir\PP^n$ be the induced map. Assume $x_0=x$, $x_i$ maps to $x_0$ and that $f_i$ is singular at $x_i$.

Then for each $i$, there is an effective divisor $A_i$ on the normalised graph $Y^\nu$, supported on the irreducible components of $Y^\nu_x$, such that
$$
\deg_{\mathcal{O}_{Y^\nu}(1)}A_i
=\delta_{x_i}(f_i)
=\frac{1}{\lambda_{x_i}^1(f_i)},
$$
In particular,
$$
\{\frac{1}{\lambda_{x_i}^1(f_i)}\}
\subseteq\operatorname{DSpec}_x(f).
$$
\end{thm}

\begin{proof}
Let $Y_i^\nu$ be the normalised graph of $f_i$. The morphism $X_i\to X$ induces a morphism
$$
\rho_i\colon Y_i^\nu\to Y^\nu
$$
such that
$$
\mathcal{O}_{Y_i^\nu}(1)=\rho_i^*\mathcal{O}_{Y^\nu}(1).
$$
Set
$$
A_i:={\rho_i}_*[Y^\nu_{i,x_i}].
$$
Then $A_i$ is an effective divisor supported on $Y^\nu_x$, and the projection formula gives
$$
\deg_{\mathcal{O}_{Y^\nu}(1)}A_i
=\deg_{\mathcal{O}_{Y_i^\nu}(1)}[Y^\nu_{i,x_i}]
=\delta_{x_i}(f_i).
$$
The remaining equality follows from Lemma~\ref{l-rational-maps-mult=fib-degree=inverse-can-threshold}.
\end{proof}

%%%%%%%%%%%%%%%%%%%%%%%%%%%%%%
\subsection{Analogues of hypersurface singularities}

As observed in Example~\ref{exa-along-hypersurface-sing}, maps to $\PP^1$ play the role of hypersurfaces among rational maps. The next lemma characterises when the polarised graph of a map to $\PP^n$ is already determined by a pencil. For regular functions $g,h$ near a smooth point, write $\widetilde g,\widetilde h$ for the functions obtained by dividing by their greatest common divisor (which makes sense because $\mathcal{O}_{X,x}$ is a UFD for a smooth variety $X$).

\begin{lem}[Reduction to a pencil]\label{l-rational-maps-fibre-P1-deg-1}
Let 
$$
f\colon X\bir\PP^n, \qquad f=(f_0:\cdots:f_n)
$$ 
be a rational map from a smooth surface, and $x\in X$ a closed point,
where the $f_i$ are regular at $x$. Assume that $f$ is singular at $x$, and let
$$
\pi\colon Y\to X
$$
be its graph. After shrinking $X$ around $x$, the following are equivalent:
\begin{enumerate}
\item the scheme-theoretic fibre $Y_x\simeq\PP^1$ and
$$
\deg\mathcal{O}_Y(1)|_{Y_x}=1;
$$
\item for some $i,j$, the polarised graph $(Y,\mathcal{O}_Y(1))$ is the polarised graph of
$$
(f_i:f_j)\colon X\bir\PP^1,
$$
equivalently, the blowup of the ideal $(f_i,f_j)$ with its tautological line bundle;
\item for some $i,j$, the graph $Y$ is the hypersurface
$$
\widetilde f_j u-\widetilde f_i v=0
$$
in $X\times\PP^1$, and $\mathcal{O}_Y(1)$ is the pullback of $\mathcal{O}_{\PP^1}(1)$.
\end{enumerate}
\end{lem}

\begin{proof}
Assume (1). The morphism
$$
Y_x\to\PP^n
$$
induced by the second projection is an isomorphism onto a line. Choose coordinates $t_i,t_j$ such that the projection
$$
\PP^n\bir\PP^1,\qquad (t_0:\cdots:t_n)\mapsto(t_i:t_j),
$$
is regular along this line and restricts to an isomorphism. The corresponding sections generate $\mathcal{O}_Y(1)$ near $Y_x$, hence, after shrinking $X$, define a morphism
$$
Y\to X\times\PP^1.
$$
Let $Z$ be its image. The induced morphism $Y\to Z$ is an isomorphism outside the fibre $Y_x$ and also an isomorphism when restricted to $Y_x$. It is therefore proper and quasi-finite, hence finite, and Lemma~\ref{lem-finite-morphism-fibrewise-closed-immersion} shows that it is an isomorphism. Since $Z$ is the graph of $(f_i:f_j)$, (2) follows.

The equivalence of (2) and (3) follows from the description of the blowup of a two-generated ideal on a smooth surface: after dividing $f_i,f_j$ by their greatest common divisor, its graph in $X\times\PP^1$ is defined by
$$
\widetilde f_j u-\widetilde f_i v=0.
$$

Finally, assume (3). Since $f$ is singular at $x$, both $\widetilde f_i$ and $\widetilde f_j$ vanish at $x$. Hence
$$
Y_x={x}\times\PP^1
$$
scheme-theoretically, and
$$
\mathcal{O}_Y(1)|_{Y_x}\simeq\mathcal{O}_{\PP^1}(1).
$$
Thus (1) holds.
\end{proof}

%%%%%%%%%%%%%%%%%%%%%%%%%%%%%%
\subsection{Linear type maps on smooth surfaces}

Linear type maps on surfaces are analogues of surface canonical singularities. The next theorem illustrates this analogy. 

\begin{thm}\label{t-linear-type-maps-smooth-surfaces}
Let 
$$
f\colon X\bir\PP^n, \qquad f=(f_0:\cdots:f_n)
$$ 
be a rational map from a smooth surface, and $x\in X$ a closed point,
where the $f_i$ are regular at $x$. Assume that $f$ is singular at $x$, and let
$$
\pi\colon Y\to X
$$
be its graph. After shrinking $X$ around $x$, the following are equivalent:
\begin{enumerate}
\item $f$ is linear type at $x$;
\item for some $i,j$, the polarised graph $(Y,\mathcal{O}_Y(1))$ is isomorphic to the hypersurface
$$
\widetilde f_j u-\widetilde f_i v=0
$$
in $X\times\PP^1$, where $\mathcal{O}_Y(1)$ is the pullback of $\mathcal{O}_{\PP^1}(1)$, and
$$
\operatorname{ord}_x(\widetilde f_i)=1
\quad\text{or}\quad
\operatorname{ord}_x(\widetilde f_j)=1;
$$
\item $Y$ has canonical singularities,
$$
Y_x\simeq\PP^1,
\qquad
\mathcal{O}_Y(1)|_{Y_x}\simeq\mathcal{O}_{\PP^1}(1);
$$
\item $Y$ has canonical singularities, $-K_Y$ is ample over $X$, and
$$
\deg\mathcal{O}_Y(1)|_{Y_x}=1;
$$
\item $\lambda_x^1(f)=1$;
\item $f^*H$ is smooth at $x$ for a general hyperplane $H\subset\PP^n$.
\end{enumerate}
Moreover, when these conditions hold, $Y$ has at most one singular point over $x$, and this is of type $A$.
\end{thm}

\begin{proof}
After shrinking $X$, we may assume that $f$ is regular away from $x$.

$(1)\implies(2)$. Let $\nu\colon Y^\nu\to Y$ be the normalisation. Since $f$ is linear type,
$$
Y^\nu_x\simeq\PP^1,
\qquad
\mathcal{O}_{Y^\nu}(1)|_{Y^\nu_x}\simeq\mathcal{O}_{\PP^1}(1).
$$
Hence the induced morphism $Y^\nu_x\to\PP^n$ is an isomorphism onto a line, so $Y^\nu_x\to Y_x$ is a closed immersion. Over $X\setminus\{x\}$, the morphism $\nu$ is an isomorphism. Thus $\nu$ is fibrewise a closed immersion, and Lemma~\ref{lem-finite-morphism-fibrewise-closed-immersion} shows that $\nu$ is an isomorphism.

Lemma~\ref{l-rational-maps-fibre-P1-deg-1} gives the hypersurface description in (2). If both $\widetilde f_i$ and $\widetilde f_j$ had order at least two at $x$, the hypersurface would be singular along $Y_x$, contradicting the normality of $Y$. Since neither has order zero, one of them has order one.

$(2)\implies(3)$. Assume that $\operatorname{ord}_x(\widetilde f_i)=1$. In the completed local ring at $x$, 
choose regular parameters $z,y$ at $x$ such that
$
\widetilde f_i=z.
$
Since $\widetilde f_i$ and $\widetilde f_j$ are coprime, the image of $\widetilde f_j$ in the DVR
$$
\widehat{\mathcal{O}}_{X,x}/(z)
$$
is non-zero: since $R\to\widehat R$ is faithfully flat,
$$
(z)\widehat{R}\cap R=(z);
$$
thus, if $\widetilde f_j$ belonged to $(z)\widehat R$, it would already belong to $(z)$ in $R$, contradicting the coprimeness of $\widetilde f_i$ and $\widetilde f_j$.

As $\widetilde f_j$ vanishes at $x$, this image is of the form
$
\overline{b}\,\overline{y}^{\,m}
$
for some $m\geq1$ and some unit $\overline{b}$. Lifting $\overline{b}$ to a unit $b\in\widehat{\mathcal{O}}_{X,x}$, we obtain
$$
\widetilde f_j=az+b y^m
$$
for some $a\in\widehat{\mathcal{O}}_{X,x}$.

On the chart $u=1$, the equation of $Y$ is formally equivalent to
$$
zv-y^m=0.
$$
Thus $Y$ is smooth if $m=1$, and otherwise has one singularity of type $A_{m-1}$. Moreover,
$$
Y_x=\{x\}\times\PP^1
$$
scheme-theoretically and
$$
\mathcal{O}_Y(1)|_{Y_x}\simeq\mathcal{O}_{\PP^1}(1).
$$

$(3)\implies(4)$. Put $E=Y_x$ which is the only exceptional curve of $\pi$. Since $Y$ is canonical, it is $\mathbb{Q}$-factorial. As $E$ is the only exceptional curve of $\pi$, we can write
$$
K_Y=\pi^*K_X+aE
$$
for some $a>0$. Since $E^2<0$ and $\rho(Y/X)=1$, the divisor $-K_Y$ is ample over $X$.

$(4)\implies(5)$. Since $Y$ is normal,
$$
\delta_x(f)=\deg\mathcal{O}_Y(1)|_{Y_x}=1.
$$
Lemma~\ref{l-rational-maps-mult=fib-degree=inverse-can-threshold} gives
$$
\lambda_x^1(f)=1.
$$

$(5)\implies(6)$. By the same lemma,
$$
\mu_xf^*H=\frac{1}{\lambda_x^1(f)}=1,
$$
so $f^*H$ is smooth at $x$.

$(6)\implies(1)$. Put $S=f^*H$, and take a minimal sequence of point blowups resolving $f$:
$$
W=X_l\longrightarrow\cdots\longrightarrow X_1\longrightarrow X_0=X.
$$
Let $S_i$ be the birational transform of $S$ on $X_i$, and let $E_i$ be the exceptional curve created on $X_i$. Every singular point of the induced map $X_i\bir \PP^n$ lies on $S_i$. Since $S$ is smooth at $x$, all the blowups are single type and the exceptional locus is a chain
$$
E_1-\cdots-E_l
$$
with
$$
E_i^2=-2\quad\text{for }i<l,
\qquad
E_l^2=-1.
$$
If $h\colon W\to\PP^n$ is the induced morphism, then $h^*H=S_l$, which is disjoint from $E_1,\ldots,E_{l-1}$ and meets $E_l$ transversally in one point. Hence $W\to Y^\nu$ contracts precisely $E_1,\ldots,E_{l-1}$ and maps $E_l$ with degree one.

Choose another general hyperplane $H'\subset\PP^n$. After shrinking $X$, the divisors $h^*H$ and $h^*H'$ have no common point near the exceptional locus and define a morphism $W\to\PP^1$. Let $Z$ be the graph of the induced pencil on $X$. Since $S$ is smooth at $x$, the argument in $(2)\implies(3)$ shows that $Z$ is normal. Both $W\to Y^\nu$ and $W\to Z$ contract precisely $E_1,\ldots,E_{l-1}$, hence
$$
Y^\nu\simeq Z.
$$
Therefore
$$
Y^\nu_x\simeq\PP^1,
\qquad
\mathcal{O}_{Y^\nu}(1)|_{Y^\nu_x}\simeq\mathcal{O}_{\PP^1}(1),
$$
so $f$ is linear type at $x$.
\end{proof}

%%%%%%%%%%%%%%%%%%%%%%%%%%%%%%
\subsection{Examples}

\begin{exa}
Recall the rational map
$$
f\colon\A^2\bir\PP^1,
\qquad
f=(x:y^n),
\qquad
n\ge 1
$$
in Example \ref{exa-rational-map-(x : y^n)}. This is of linear type.
\end{exa}

%%%%%%%%%%%%%%%%%%%%%%%%%%%%%%
%%%%%%%%%%%%%%%%%%%%%%%%%%%%%%
\section{\bf Linear type maps on singular surfaces}

In this section we study the existence of non-trivial linear type maps on surfaces. On smooth surfaces such maps arise naturally from maximal ideal maps at smooth points. On singular surfaces, however, their existence is much more subtle, and they need not exist even for arbitrary canonical surface singularities. We first treat affine toric surfaces before moving onto other classes of surface singularities.

%%%%%%%%%%%%%%%%%%%%%%%%%%%%%%
\subsection{Toric surfaces}

The following theorem gives a criterion for a toric map to be linear type and shows that every affine toric surface admits such a non-trivial map.

\begin{thm}[Linear type maps on toric surfaces]
\label{t-linear-type-maps-on-toric-surfaces}
Let $N=\Z^2$ be a lattice of rank two, let
$$
\sigma=\langle v_0,v_1\rangle\subset N_{\mathbb R}
$$
be a strongly convex rational polyhedral cone, and let
$$
X=X_\sigma,\qquad x=x_\sigma.
$$
Let
$$
\alpha\colon N\longrightarrow\mathbb Z
$$
be a nonzero homomorphism such that $\alpha(v_0)$ and $\alpha(v_1)$ have opposite signs, and let $w\in\operatorname{Int}(\sigma)\cap N$ be the primitive generator of $\ker(\alpha)$. Recall
$$
\mu_\sigma(w)
=
\min\{\langle m,w\rangle>0\mid m\in\sigma^\vee\cap M\}
$$
and
$$
d(\alpha)=[\mathbb Z:\alpha(N)]
$$
from Theorem \ref{thm:toric-normalised-graph-fibre-degree}.
Then the induced toric rational map
$$
f\colon X\dashrightarrow\mathbb P^1
$$
is singular only at $x$, and it is linear type at $x$ if and only if
$$
\delta_x(f)=1, \ i.e. \ \mu_\sigma(w)=d(\alpha)=1.
$$
\end{thm}

\begin{proof}
Let
$$
p\colon Y^\nu\longrightarrow X,\qquad
q\colon Y^\nu\longrightarrow\mathbb P^1
$$
be the normalised graph which is given by subdividing $\sigma$ by inserting $\R_{\ge 0}w$. Let $E=D_w$ be the exceptional curve. By Theorem~\ref{thm:toric-normalised-graph-fibre-degree},
$$
[Y^\nu_x]=\mu_\sigma(w)E,
$$
and
$$
E\simeq\mathbb P^1,\qquad
\mathcal{O}_{Y^\nu}(1)|_E
\simeq
\mathcal{O}_{\mathbb P^1}(d(\alpha)).
$$
Thus linear type implies
$$
\delta_x(f)=\mu_\sigma(w)d(\alpha)=1, \qquad \mu_\sigma(w)=d(\alpha)=1.
$$

Conversely, assume these equalities hold. Then $\alpha$ is surjective, hence is primitive (i.e. primitive when viewed as an element of the dual lattice $M$) and $q|_E$ is an isomorphism. 

We show the fibre $Y^\nu_x$ is reduced. 
More precisely, we show 
$$
m_x\mathcal{O}_{Y^\nu}=I_E
$$
where the left hand side is the ideal sheaf of $Y^\nu_x$ and the right hand side is the ideal sheaf of $E$. It is clear that 
$$
m_x\mathcal{O}_{Y^\nu}\subseteq I_E.
$$
We will show the opposite inclusion.

Possibly switching $v_0,v_1$, we can assume
$$
\alpha(v_0)<0<\alpha(v_1).
$$
Since $\mu_\sigma(w)=1$, there is $m\in\sigma^\vee\cap M$ such that
$$
\langle m,w\rangle=1.
$$

 Consider the chart
$$
U_0=X_{\langle v_0,w\rangle}\subset Y^\nu.
$$
Then by toric geometry \cite[Page 53]{Fulton-toric-varieties}\cite[Proposition 4.3.3]{Cox-Little-Schenck}, $I_E$ on this chart is given by the ideal generated by
$$
\{\chi^u \mid u\in\langle v_0,w\rangle^\vee\cap M,
\ \langle u,w\rangle>0\}. 
$$
Pick $u$ as in this set and let 
$$
r:=\langle u,w\rangle>0.
$$
If $u\in\sigma^\vee$, then $\chi^u\in m_x$, hence $\chi^u\in m_x\mathcal{O}_{U_0}$. Assume $u\notin\sigma^\vee$ in which case
$$
\langle u,v_1\rangle<0.
$$
We will show that again $\chi^u\in m_x\mathcal{O}_{U_0}$.
Since 
$$
\langle rm-u,w\rangle=0, \ \mbox{and $\alpha$ is primitive}, 
$$
there is an integer $k$ such that
$$
rm=u+k\alpha.
$$
Moreover, $k>0$ because $\langle rm,v_1\rangle\ge 0$, $\langle u,v_1\rangle<0$ while $\alpha(v_1)>0$.

Since $r>0$,
$$
u+k\alpha=rm\in\sigma^\vee\cap M.
$$
Moreover,
$$
-k\alpha\in\langle v_0,w\rangle^\vee\cap M,
$$
so
$$
\chi^u
=
\chi^{u+k\alpha}\chi^{-k\alpha}
\in
m_x\mathcal{O}_{U_0}.
$$
Hence
$$
m_x\mathcal{O}_{U_0}=I_{E\cap U_0}.
$$

Similar arguments apply on the other chart 
$$
U_1=X_{\langle w,v_1\rangle}.
$$
Indeed, $I_E$ on this chart is given by the ideal generated by
$$
\{\chi^t \mid t\in\langle w,v_1\rangle^\vee\cap M,
\ \langle t,w\rangle>0\}. 
$$
Pick $t$ as in this set and let 
$$
s:=\langle t,w\rangle>0.
$$
If $t\in\sigma^\vee$, then $\chi^t\in m_x$, hence $\chi^t\in m_x\mathcal{O}_{U_1}$. Assume $t\notin\sigma^\vee$ in which case
$$
\langle t,v_0\rangle<0.
$$
We will show that again $\chi^t\in m_x\mathcal{O}_{U_1}$.
Since 
$$
\langle sm-t,w\rangle=0, \ \mbox{and $\alpha$ is primitive}, 
$$
there is an integer $l$ such that
$$
sm=t+l\alpha.
$$
Moreover, $l<0$ because $\langle sm,v_0\rangle\ge 0$, $\langle t,v_0\rangle<0$ while $\alpha(v_0)<0$.

Since $s>0$,
$$
t+l\alpha=sm\in\sigma^\vee\cap M.
$$
Moreover,
$$
-l\alpha\in\langle w,v_1\rangle^\vee\cap M,
$$
so
$$
\chi^t
=
\chi^{t+l\alpha}\chi^{-l\alpha}
\in
m_x\mathcal{O}_{U_1}.
$$
Hence
$$
m_x\mathcal{O}_{U_1}=I_{E\cap U_1}.
$$

Therefore
$$
m_x\mathcal{O}_{Y^\nu}=I_E,
$$
so
$$
Y^\nu_x=E\simeq\mathbb P^1
$$
scheme-theoretically. Since $q|_E$ is an isomorphism,
$$
\mathcal{O}_{Y^\nu}(1)|_{Y^\nu_x}
\simeq
\mathcal{O}_{\mathbb P^1}(1).
$$
Thus $f$ is linear type at $x$.
\end{proof}

\begin{cor}[Toric surfaces admit singular linear type maps]\label{c-toric-surfaces-admit-linear-type-map}
Every affine toric surface associated to a two-dimensional strongly convex cone admits a toric singular linear type map to $\mathbb P^1$ at its torus-fixed closed point.
\end{cor}
\begin{proof}
Say 
$$
\sigma=\langle v_0,v_1\rangle\subset N_{\mathbb R},
\qquad  X=X_\sigma, \qquad x=x_\sigma.
$$ 
We can assume $X$ is singular at $x$. 
Let $m\in M$ be the primitive generator of the ray of $\sigma^\vee$ vanishing on $v_0$. Choose $n\in N$ such that
$$
\langle m,n\rangle=1.
$$
For $l\gg0$, set
$$
w=n+lv_0.
$$
Then
$$
w\in\operatorname{Int}(\sigma),
\qquad
\langle m,w\rangle=1.
$$
In particular, $w$ is primitive. Let
$$
\alpha\colon N\longrightarrow N/\mathbb Zw\simeq\mathbb Z
$$
be the quotient homomorphism. Then $\alpha(v_0), \alpha(v_1)$ have opposite signs, and
$$
d(\alpha)=1, \qquad \mu_\sigma(w)=1.
$$
Hence the induced map $f\colon X\bir \PP^1$ is linear type. 
\end{proof}

%%%%%%%%%%%%%%%%%%%%%%%%
\subsection{Toric examples}

\begin{exa}[Surface toric projections]
Let
$$
N=\mathbb Z^2,\qquad w=(0,1),
$$
and let $\sigma\subset N_{\mathbb R}$ be a two-dimensional cone with $w\in\operatorname{Int}(\sigma)$. For $r\geq1$, consider
$$
\alpha_r\colon N\longrightarrow\mathbb Z,\qquad
\alpha_r(a,b)=ra.
$$
Let
$$
f_r\colon X_\sigma\dashrightarrow\mathbb P^1
$$
be the induced toric rational map, and let $E=D_w$. Then
$$
[Y^\nu_x]=\mu_\sigma(w)E,
\qquad
\mathcal{O}_{Y^\nu}(1)|_E
\simeq
\mathcal{O}_{\mathbb P^1}(r),
$$
and hence
$$
\delta_x(f_r)=r\mu_\sigma(w).
$$
Thus $f_r$ is linear type if and only if
$$
r=\mu_\sigma(w)=1.
$$
\end{exa}

\begin{exa}[A primitive projection which is not linear type]
Let
$$
\sigma=\langle(5,-1),(-2,1)\rangle\subset\mathbb R^2,
\qquad
w=(0,1),
$$
and consider the homomorphism
$$
\alpha\colon\mathbb Z^2\longrightarrow\mathbb Z,\qquad
\alpha(a,b)=a.
$$
Then $m=(s,t)\in M\cap \sigma^\vee$ iff
$$
2s\leq t\leq5s.
$$
On the other hand, $\langle m,w\rangle=t$, hence there is no $m\in M\cap \sigma^\vee$ with $\langle m,w\rangle=1$. 
But 
$$
m=(1,2)\in M\cap \sigma^\vee,
$$ 
so  
$
\mu_\sigma(w)=2.
$
Therefore, 
$$
\delta_x(f)=2.
$$
Thus $f$ is not linear type, even though $\alpha$ is primitive.
\end{exa}

\begin{exa}[Linear type maps from canonical surfaces with deeply singular normalised graphs]
Let $r\geq 2$, let $N=\mathbb Z^2$, and consider
$$
\sigma=\langle (1,0),(1,r)\rangle\subset N_{\mathbb R}.
$$
Let $X=X_\sigma$ and let $x=x_\sigma$ be its torus-fixed point. The surface $X$ has a canonical singularity at $x$; in fact, it is of type $A_{r-1}$.

Consider the homomorphism
$$
\alpha\colon N\longrightarrow\mathbb Z,\qquad
\alpha(a,b)=a-rb.
$$
We have
$$
\alpha(1,0)=1,\qquad
\alpha(1,r)=1-r^2<0,
$$
and the primitive generator of $\ker(\alpha)$ in $\operatorname{Int}(\sigma)$ is
$$
w=(r,1).
$$
Moreover, $d(\alpha)=1$, and $(0,1)\in\sigma^\vee\cap M$ satisfies
$$
\langle(0,1),w\rangle=1.
$$
Hence $\mu_\sigma(w)=1$, so by Theorem~\ref{t-linear-type-maps-on-toric-surfaces}, the induced toric rational map
$$
f\colon X\dashrightarrow\mathbb P^1
$$
is linear type at $x$.

On the other hand, the normalised graph $Y^\nu$ is obtained by subdividing $\sigma$ along $\mathbb R_{\geq0}w$. Its maximal cones are
$$
\langle(1,0),(r,1)\rangle,\qquad
\langle(r,1),(1,r)\rangle.
$$
The first is smooth. On the second cone the toric log discrepancy function is
$$
\psi(a,b)=\frac{a+b}{r+1}.
$$
The lattice point $(1,1)$ lies in its interior, and every interior lattice point $(a,b)$ has $a,b\geq1$. Therefore
$$
\theta_x(f)=\mld(Y^\nu,0)=\frac{2}{r+1}.
$$
Thus $X$ is canonical and $f$ is linear type, while the singularities of $Y^\nu$ become arbitrarily deep klt as $r\to\infty$.
\end{exa}

This example is in sharp contrast with the case when $X$ is smooth in which case $Y^\nu$ has canonical singularities, by Theorem \ref{t-linear-type-maps-smooth-surfaces}.

%%%%%%%%%%%%%%%%%%%%%%%%%%%%%%
\subsection{Rational singularities: degree one implies linear type}

\begin{thm}
\label{t-degree-one-linear-type-rational-surfaces}
Let $X$ be a normal surface and $x\in X$ a smooth or rational singularity. Let
$$
f\colon X\dashrightarrow \mathbb P^n
$$
be a rational map singular at $x$. Then $f$ is linear type at $x$ if and only if
$$
\delta_x(f)=1.
$$
\end{thm}

\begin{proof}
If $f$ is linear type at $x$, then
$$
\bigl(Y^\nu_x,\mathcal O_{Y^\nu}(1)|_{Y^\nu_x}\bigr)
\simeq
\bigl(\mathbb P^1,\mathcal O_{\mathbb P^1}(1)\bigr),
$$
hence $\delta_x(f)=1$.

Conversely, assume that $\delta_x(f)=1$. If $x\in X$ is smooth, then $f$ is linear type by Theorem \ref{t-linear-type-maps-smooth-surfaces}, hence we can assume $X$ is singular at $x$.

After shrinking $X$, we may assume that $X$ is affine and that $x$ is the only singular point of $f$ and of $X$. Let
$$
p\colon Y^\nu\longrightarrow X,
\qquad
q\colon Y^\nu\longrightarrow\mathbb P^n
$$
be the normalised graph. Write
$$
[Y^\nu_x]=\sum_i a_iE_i,
\qquad
a_i=\operatorname{ord}_{E_i}(m_x).
$$
Since $\mathcal O_{Y^\nu}(1)$ is $p$-ample, every $\mathcal O_{Y^\nu}(1)\cdot {E_i}$ is a positive integer. Therefore
$$
1=\delta_x(f)=\mathcal O_{Y^\nu}(1)\cdot[Y^\nu_x]
=\sum_i a_i\mathcal O_{Y^\nu}(1)\cdot E_i
$$
implies that
$$
[Y^\nu_x]=E
\qquad\text{and}\qquad
\mathcal O_{Y^\nu}(1)\cdot E=1
$$
for a unique exceptional prime divisor $E$.

We show that $Y^\nu_x$ is reduced. Take a resolution
$$
h\colon W\longrightarrow Y^\nu
$$
and put $\phi=p\circ h$. Since $x\in X$ is a rational singularity,
$$
R^1\phi_*\mathcal O_W=0,
$$
and so we get 
$$
H^1(W,\mathcal{O}_W)=0
$$
because $X$ is affine.
Moreover, $h_*\mathcal O_W=\mathcal O_{Y^\nu}$ as $Y^\nu$ is normal. The Leray spectral sequence therefore gives an exact sequence 
$$
0 \to H^1(Y^\nu,\mathcal{O}_{Y^\nu}=h_*\mathcal O_W) \to H^1(W,\mathcal{O}_W) \to H^0(Y^\nu,R^1 h_*\mathcal{O}_W) \to \cdots
$$
 Thus 
$$
H^1(Y^\nu,\mathcal O_{Y^\nu})=0.
$$

Consider the base change of $Y^\nu\to X$ to $\Spec \mathcal{O}_{X,x}$ to get 
$$
g\colon V=\Spec \mathcal{O}_{X,x}\times_XY^\nu\to \Spec \mathcal{O}_{X,x}.
$$
Then by flat base change,
$$
H^1(V,\mathcal O_{V})=0.
$$
Now Lipman's divisoriality \cite[Proposition 3.1]{Lipman-rational-singularities} shows that $m_x\mathcal{O}_V$ is a reflexive rank one ideal sheaf. Its order along $E$ is one, and its order along every other prime divisor on $V$ is zero. Thus since $V$ is normal we get 
$$
m_x\mathcal O_{V}=\mathcal O_{V}(-E)=\mathcal I_E
$$
where the right hand side is the ideal sheaf of $E$ in $V$.
Therefore the fibre
$$
V_x=E
$$
scheme-theoretically which in turn gives
$$
Y^\nu_x=E.
$$

Finally, $q|_E$ is nonconstant because $\mathcal O_{Y^\nu}(1)\cdot E=1$. If $C=q(E)$ with its reduced structure, then
$$
1=\mathcal O_{Y^\nu}(1)\cdot E=\deg(q|_E)\deg C.
$$
Hence $\deg(q|_E)=\deg C=1$. Thus $C$ is a line and the finite birational morphism
$$
q|_E\colon E\longrightarrow C
$$
is an isomorphism. Consequently,
$$
E\simeq\mathbb P^1,
\qquad
\mathcal O_{Y^\nu}(1)|_E\simeq\mathcal O_{\mathbb P^1}(1).
$$
Hence $f$ is linear type at $x$.
\end{proof}

%%%%%%%%%%%%%%%%%%%%%%%%%%
\subsection{Degree one maps and smooth curves}

The next lemma helps when we try to apply Theorem \ref{t-degree-one-linear-type-rational-surfaces}.

\begin{thm}
\label{t-degree-one-maps-smooth-curves}
Let $x\in X$ be a normal surface germ. Then $x\in X$ admits a rational map
$$
f\colon X\dashrightarrow \mathbb P^n
$$
singular at $x$ with
$$
\delta_x(f)=1
$$
if and only if there is a curve
$$
D\subset X
$$
through $x$ which is smooth at $x$. Moreover, in the latter case one can take $n=1$.
\end{thm}

\begin{proof}
Suppose first that $\delta_x(f)=1$. For a general hyperplane
$$
H\subset\mathbb P^n,
$$
Theorem \ref{t-norm-graph-fibre-degree-mult-hyperplane-pullback} gives
$$
\mu_xf^*H=\delta_x(f)=1.
$$
Since $f^*H$ is reduced, it is smooth at $x$.

Conversely, suppose that $D\subset X$ is smooth at $x$. Take a resolution
$$
\phi\colon W\longrightarrow X
$$
on which
$$
m_x\mathcal{O}_W=\mathcal{O}_W(-N)
$$
for an effective exceptional divisor $N$, and let $C\subset W$ be the birational transform of $D$. The induced morphism
$$
C\longrightarrow D
$$
is proper and birational, hence finite birational, so an isomorphism because $D$ is smooth. If $w\in C$ is the point mapping to $x$, then
$$
\mathcal{O}_W(-N)|_C
=
m_x\mathcal{O}_C
=
\mathcal{O}_C(-w).
$$
Thus
$$
N\cdot C=1.
$$
In particular, $C$ meets the exceptional locus at the single point $w$, where it meets a unique exceptional curve transversally, and this exceptional curve has coefficient $1$ in $N$.

After shrinking $X$, we may assume that $X$ is affine. The evaluation morphism
$$
\phi^*\phi_*\mathcal{O}_W(C)\longrightarrow\mathcal{O}_W(C)
$$
is surjective away from the exceptional locus. Hence we can choose a finite-dimensional subsystem of $|C|$, containing $C$, which is free away from the exceptional locus. Let
$$
C'\sim C
$$
be a general member. We may assume that $C'$ contains neither $C$ nor any exceptional curve. By Bertini, $C'$ is smooth away from the exceptional locus. Moreover,
$$
N\cdot C'=N\cdot C=1.
$$
It follows that $C'\cap N$ consists of a single reduced point $w'$. In particular, $C'$ is smooth at $w'$ and meets a unique exceptional curve transversally there. After shrinking $X$ again, we may assume that $C$ and $C'$ have no common point away from the exceptional locus.

The divisors $C$ and $C'$ determine a pencil. If $w\neq w'$, the pencil has no base point over $x$. If $w=w'$, then $w$ is its unique base point over $x$, so resolving this pencil by a finite sequence of point blowups and replacing $W$ with the resulting surface we can assume that $C,C'$ do not intersect, hence that the pencil is base point free. Replacing $N$ with its pullback the condition $N\cdot C'=N\cdot C=1$ is preserved.

Now the pencil defines a morphism
$$
h\colon W\longrightarrow\mathbb P^1
$$
with $C,C'$ being two of its distinct fibres.
Let
$$
f\colon X\dashrightarrow\mathbb P^1
$$
be the rational map induced by $h$. The morphism
$$
(\phi,h)\colon W\longrightarrow X\times\mathbb P^1
$$
factors through the normalised graph
$$
r\colon W\longrightarrow Y^\nu.
$$
As $N$ is the divisor defined by $m_x\mathcal{O}_W$, we have
$$
r_*N=[Y^\nu_x].
$$
Therefore, by the projection formula,
$$
\delta_x(f)
=
\mathcal{O}_{Y^\nu}(1)\cdot[Y^\nu_x]
=
h^*\mathcal{O}_{\mathbb P^1}(1)\cdot N
=
1.
$$
Moreover, the morphism $h$ is nonconstant on $\phi^{-1}(x)$, hence $f$ is singular at $x$.
\end{proof}

%%%%%%%%%%%%%%%%%%%%%%%%%%
\subsection{Canonical singularities and local class groups}

In this subsection assume $k=\mathbb C$. 

\begin{cor}
\label{c-canonical-linear-type-local-class-group}
Let $x\in X$ be a singular canonical surface germ. Then
$$
x\in X\text{ admits a singular linear type map at }x
\quad\Longleftrightarrow\quad
\operatorname{Cl}(\mathcal O_{X,x})\neq0.
$$
\end{cor}
\begin{proof}
Since $x\in X$ is rational, Theorems
\ref{t-degree-one-linear-type-rational-surfaces} and
\ref{t-degree-one-maps-smooth-curves} show that $x\in X$ admits a singular linear type map at $x$ if and only if there is a curve
$$
D\subset X
$$
through $x$ which is smooth at $x$.

Suppose first that
$$
\operatorname{Cl}(\mathcal O_{X,x})\neq0.
$$
By \cite[Proposition 4.9]{Brevik-Nollet-Srinivas}, any nonzero class in
$\operatorname{Cl}(\mathcal O_{X,x})$ is represented by a curve smooth at $x$. Since its class is nonzero, this curve necessarily passes through $x$.

Conversely, if $D\subset X$ is smooth at $x$, then its class is nonzero. Indeed, otherwise $D$ would be Cartier near $x$, but a smooth Cartier curve through a singular surface point would force $x\in X$ to be smooth.
\end{proof}

This also shows that the existence of a singular linear type map is not determined formally. Indeed, let $B$ be the completed local ring of a canonical surface singularity. By \cite[Theorem 1.4]{Brevik-Nollet-Srinivas}, for every subgroup
$$
H\subset\operatorname{Cl}(B)
$$
there is a normal geometric local domain $A$ such that
$$
\widehat A\simeq B
\qquad\text{and}\qquad
\operatorname{Cl}(A)=H\subset\operatorname{Cl}(B).
$$
Thus, whenever $\operatorname{Cl}(B)\neq0$, there are algebraic realisations with completion $B$ such that one admits a singular linear type map and another does not. 

The local class groups of the completions of the canonical surface singularities are
$$
\operatorname{Cl}(\widehat{\mathcal O}_{A_n})
\simeq
\mathbb Z/(n+1)\mathbb Z,
$$
$$
\operatorname{Cl}(\widehat{\mathcal O}_{D_n})
\simeq
\begin{cases}
\mathbb Z/4\mathbb Z,& n\text{ odd},\\
\mathbb Z/2\mathbb Z\oplus\mathbb Z/2\mathbb Z,& n\text{ even},
\end{cases}
$$
and
$$
\operatorname{Cl}(\widehat{\mathcal O}_{E_6})
\simeq
\mathbb Z/3\mathbb Z,
\qquad
\operatorname{Cl}(\widehat{\mathcal O}_{E_7})
\simeq
\mathbb Z/2\mathbb Z,
\qquad
\operatorname{Cl}(\widehat{\mathcal O}_{E_8})=0.
$$
See \cite[\S4]{Brevik-Nollet-Srinivas}.

Therefore, if $B$ is the completed local ring of a canonical singularity of either types 
$$
A_n,\quad D_n,\quad E_6,\quad E_7,
$$
then taking $H=0$ gives a factorial canonical singularity $x\in X$ which admits no singular linear type map at $x$, while taking $H\neq0$ gives a non-factorial one which does. 

Since the class group of the completed $E_8$ singularity $x\in X$ is trivial, they do not admit singular linear type maps. 
There is also a direct resolution-theoretic explanation for this case. On the minimal resolution
$$
\pi\colon W\to X,
$$
the maximal-ideal cycle is the fundamental cycle
$$
Z=
2F_1+4F_2+6F_3+5F_4+4F_5+3F_6+2F_7+3F_8
$$
for a suitable numbering of the exceptional curves. Moreover, 
$$
m_x\mathcal{O}_W=\mathcal{O}_W(-Z).
$$
Thus the order of $m_x$ along every exceptional component is at least $2$. This remains true on every higher resolution.
On the other hand, a singular linear type map at $x$ would have reduced normalised graph fibre, so its exceptional curve $C$ would satisfy
$$
\operatorname{ord}_C(m_x)=1.
$$
This gives another proof that the canonical $E_8$ singularity admits no singular linear type map.

%%%%%%%%%%%%%%%%%%%%%%%%%%
\subsection{$A$-type klt singularities}\label{ss-A-type-klt-sing-lin-type}

Following the usual extension of the $ADE$ terminology, a klt surface singularity is called $A$-type if the dual graph of its minimal resolution is a chain, $D$-type if it is a fork with three branches, two of length one, and $E$-type otherwise, cf. \cite{Alexeev-two-dimensional-terminations, Moraga-coregularity}.

Every toric surface singularity admits a toric singular linear type map by Corollary \ref{c-toric-surfaces-admit-linear-type-map}. In general, let $x\in X$ be an $A$-type klt surface singularity. Such a singularity is formally toric, and the exceptional locus of its minimal resolution is a chain. Nevertheless, the existence of a singular linear type map is not determined by the formal isomorphism type. Since $x\in X$ is rational \cite{Kollar-Mori}, Theorems
\ref{t-degree-one-linear-type-rational-surfaces} and
\ref{t-degree-one-maps-smooth-curves} give
$$
x\in X\text{ admits a singular linear type map at }x
\quad\Longleftrightarrow\quad
\text{there is a curve }D\subset X\text{ smooth at }x.
$$

In particular, the canonical $A_n$ examples above (over $k=\C$) show that an $A$-type klt surface singularity may or may not admit a singular linear type map, even among algebraic germs with isomorphic completions. 

%%%%%%%%%%%%%%%%%%%%%%%%%%
\subsection{$D$-type klt singularities}\label{ss-D-type-klt-sing-lin-type}

For canonical $D_n$ singularities (over $k=\C$), Corollary
\ref{c-canonical-linear-type-local-class-group} shows that the existence of a singular linear type map depends on the algebraic local class group. For general $D$-type klt singularities, there is a similar distinction between the formal and algebraic settings.

\begin{lem}
\label{l-D-type-smooth-complement}
Let $x\in X$ be a non-canonical $D$-type klt surface singularity admiting a $2$-complement
$$
K_X+B
$$
with $B$ reduced. Then $B$ is smooth at $x$.
\end{lem}
\begin{proof}
Let
$$
\phi\colon W\longrightarrow X
$$
be the minimal resolution. The exceptional locus looks like 
$$
\begin{array}{ccccccccccc}
&F_1&&&&&&&&&\\
&|&&&&&&&&&\\
&E_1&-&E_2&-\cdots-&E_{r-1}&-&E_r.\\
&|&&&&&&&&&\\
&F_2&&&&&&&&&
\end{array}
$$
If $C$ is the birational transform of $B$, then $C$ is smooth, meets $E_r$ transversally in one point, and is disjoint from the other exceptional curves.

Since $x\in X$ is rational, \cite[Proposition 3.1 and Theorem 12.1]{Lipman-rational-singularities} give
$$
m_x\mathcal O_W=\mathcal O_W(-Z),
$$
where $Z$ is the fundamental cycle. The cycle
$$
F_1+F_2+2(E_1+\cdots+E_{r-1})+E_r
$$
is anti-nef. By the minimality of $Z$,
$$
Z\leq F_1+F_2+2(E_1+\cdots+E_{r-1})+E_r.
$$
Since every exceptional component occurs in $Z$ with positive coefficient, the coefficient of $E_r$ in $Z$ is one.

If
$$
w=C\cap E_r,
$$
then
$$
m_x\mathcal O_C
=
\mathcal O_C(-Z|_C)
=
\mathcal O_C(-w).
$$
Since $C\to D$ is the normalisation near $x$, its scheme-theoretic fibre over $x$ has length one. Theorem
\ref{t-mult-normalisation-fibre-curves} therefore gives
$$
\mu_xB=1.
$$
Thus $B$ is smooth at $x$.
\end{proof}

Every $D$-type non-canonical klt surface singularity admits a 2-complement. Formally there is no obstruction to the existence of a reduced $2$-complement. Indeed, by the classification of $D$-type surface complements in \cite[\S6]{Shokurov-complements-surfaces}\cite[Proposition 3.11]{Moraga-coregularity}, for a non-canonical klt $D$-type surface singularity $x\in X$, the completed germ $\widehat X$ admits a $2$-complement
$$
K_{\widehat X}+B
$$
with $B$ reduced. The reduced $2$-complement, however, need not exist on the original algebraic germ.

\begin{exa}[A D-type singularity without a reduced $2$-complement]
\label{e-D-type-no-reduced-2-complement}
Consider the analytic $D$-type klt singularity whose minimal resolution has graph
$$
\begin{array}{ccccc}
&&F_1(-2)&&\\
&&|&&\\
F_2(-2)&-&E_1(-3)&-&E_2(-2)
\end{array}
$$
where the numbers in brackets show the self-intersections.
If
$$
\phi\colon W\longrightarrow X
$$
is the minimal resolution, then we can calculate 
$$
K_W+\frac23E_1+\frac13(F_1+F_2+E_2)=\phi^*K_X.
$$
Thus the singularity is klt but non-canonical, and $K_X$ has Cartier index $3$.

By \cite[Theorem 1.1]{Parameswaran-vanStraten}, there is an algebraic realisation (denotes again by $x\in X$) of this analytic germ such that
$$
\operatorname{Cl}(\mathcal O_{X,x})
=
\langle K_X\rangle
\simeq
\mathbb Z/3\mathbb Z.
$$
We claim that this algebraic germ admits no reduced $2$-complement. Indeed, if
$$
K_X+B
$$
were such a complement, then
$$
2[K_X+B]=0
$$
in $\operatorname{Cl}(\mathcal O_{X,x})$. Since this group has order $3$, we get
$$
K_X+B\sim0.
$$
Thus $K_X+B$ would be a $1$-complement.

There is no such $1$-complement. Indeed, since $K_X+B$ is Cartier and $(X,B)$ is lc, its log discrepancies at the exceptional curves on the minimal resolution would be $0$. Hence 
$$
K_W+C+E_1+E_2+F_1+F_2=\phi^*(K_X+B),
$$
where $C$ is the birational transform of $B$. Intersecting with $E_1$ gives
$$
1\le C\cdot E_1+1=0,
$$
a contradiction.
\end{exa}

Thus the existence of the standard reduced $2$-complement is not determined by the formal isomorphism type. On the other hand, the existence of a singular linear type map only requires a smooth curve through the singular point. This suggests the following question.

\begin{quest}
\label{q-D-type-smooth-curve-reduced-complement}
Let $x\in X$ be a non-canonical $D$-type klt surface singularity. If there is a curve
$$
C\subset X
$$
which is smooth at $x$, does $x\in X$ admit a reduced $2$-complement?
\end{quest}

The converse implication follows from the resolution-theoretic description of reduced $D$-type $2$-complements: such a complement would give a curve smooth at $x$. Thus the question asks whether
$$
\text{$x\in X$ admits a singular linear type map}
\quad\Longleftrightarrow\quad
\text{$x\in X$ admits a reduced $2$-complement}.
$$

%%%%%%%%%%%%%%%%%%%%%%%%%%%%%%
\section{\bf Questions and future directions}
\label{s-questions-future}

The results of this paper suggest that singularities of rational maps form a birational
theory parallel in several ways to the usual theory of singularities of varieties and pairs.
Even for maps to projective space, many basic questions remain open. We mention a few
directions which seem particularly natural.

%%%%%%%%%%%%%%%%%%%%%%%%%%%%%%
\subsection{Higher dimension}

Much of the detailed theory developed here concerns maps from surfaces. In higher
dimension the geometry of the normalised graph fibre can be considerably more complicated.
For example, even when $X$ is smooth and a general hyperplane pullback $f^*H$ is smooth,
the normalised graph can have non-trivial singularities and the fibre
$
Y^\nu_x
$
need not be a projective space.

A basic problem is therefore to investigate classes of mild singularities of rational maps
in higher dimension, and numerical or birational criteria for recognising them.

\begin{quest}
Which germs $x\in X$ of fixed dimension $d$ admit a singular linear type
rational map
$$
f\colon X\dashrightarrow \mathbb P^n?
$$
More generally, can one classify rational maps for which the normalised graph fibre degree
$\delta_x(f)$ is small or the polarised normalised graph fibre is simple, e.g. a weighted projective surface.
\end{quest}

It would already be interesting to understand the cases
$$
\delta_x(f)=1,\qquad \delta_x(f)=2
$$
in dimension three. Even the case $\delta_x(f)=2$ in dimension two is not classified in this paper.  

%%%%%%%%%%%%%%%%%%%%%%%%%%%%%%
\subsection{Boundedness and anti-canonical maps}

A recurring theme in birational geometry is that quantitative control of singularities, together with suitable positivity, often forces boundedness of the underlying birational models or of natural constructions attached to them. It is natural to ask for analogous results for the
polarised normalised graph
$$
Y^\nu,\mathcal{O}_{Y^\nu}(1)\longrightarrow X.
$$

For example, one may ask when boundedness of the fibre degrees
$$
\delta_x(f)=\deg\mathcal{O}_{Y^\nu}(1)|_{Y^\nu_x}
$$
implies boundedness of the relative geometry of $Y^\nu\to X$, or bounded relative very
ampleness of $\mathcal{O}_{Y^\nu}(m)$.

There is a particularly natural version of this problem in Fano geometry. Let
$$
X\longrightarrow Z
$$
be a klt Fano contraction and, for $m>0$, consider the rational map defined by
$$
|-mK_X|/Z.
$$

\begin{quest}
Fix $d$. Is there a bounded $m$, depending only on $d$ and possibly on prescribed
singularity data, such that the map defined by $|-mK_X|/Z$ has uniformly controlled map
singularities?
\end{quest}

One can ask for different levels of control: regularity, linear type, bounded normalised
graph fibre degree, a positive lower bound for suitable thresholds, or bounded singularities
of the normalised graph.  It would be interesting to understand how complement theory and
other boundedness results for Fano varieties interact with singularities of maps.

%%%%%%%%%%%%%%%%%%%%%%%%%%%%%%
\subsection{Restriction, adjunction and composition}

The equality
$$
\delta_x(f)=\mu_xf^*H
$$
for maps from normal surfaces shows that the singularities of a rational map are closely
related to those of a general hyperplane pullback. This suggests that the theory should
have an inductive aspect.

\begin{quest}
Let
$$
f\colon X\dashrightarrow\mathbb P^n
$$
and let $S\subset X$ be a divisor. How are the singularities of $f$ along $S$ related to
those of the induced rational map on $S$?
\end{quest}

The most natural case is when $S=f^*H$ for a general hyperplane $H\subset\mathbb P^n$.
One may hope for forms of adjunction and inversion of adjunction for singularities of
rational maps.

Another basic operation is composition. Given rational maps
$$
X\dashrightarrow Y\dashrightarrow Z,
$$
it would be useful to understand how the singularities of the composite are related to
those of the two maps. 

One may also ask whether complicated map singularities admit a factorisation into
simpler ones, in analogy with the factorisation of birational maps and the minimal model
program. 

%%%%%%%%%%%%%%%%%%%%%%%%%%%%%%
\subsection{Moduli}

Suppose $X$ is projective and $L$ is a line bundle. Rational maps
$$
X\dashrightarrow\mathbb P^n
$$
defined by $L$ are parametrised, up to automorphisms of the target, by suitable
$(n+1)$-dimensional subspaces of $H^0(X,L)$. This suggests studying loci on the relevant
Grassmannian determined by the singularities of the corresponding rational maps.

\begin{quest}
Can one construct useful moduli spaces of rational maps with prescribed singularity
properties, for example linear type maps, maps with bounded fibre degree, or maps whose
normalised graphs have prescribed singularities?
\end{quest}

One may also ask for natural compactifications. The polarised graph, or its normalisation,
provides a canonical geometric object attached to a rational map and may give a useful
starting point for such compactifications.

%%%%%%%%%%%%%%%%%%%%%%%%%%%%%%
\subsection{Ordinary singularities}

There is also a natural connection in the opposite direction: rational maps may provide
new ways to study ordinary singularities. Let $x\in X$ be a normal germ and choose
generators of the maximal ideal $m_x$. They define a rational map
$$
f_x\colon X\dashrightarrow\mathbb P^n
$$
whose graph is the blowup of $m_x$.

\begin{quest}
Can we use singularities of $f_x$ and the polarised normalised graph
$$
Y^\nu,\mathcal{O}_{Y^\nu}(1)\longrightarrow X
$$
to understand the usual singularities of $x\in X$?
\end{quest}
For example, the multiplicity of $x\in X$ coincides with $\delta_x(f)$ as discussed in Example \ref{e-maximal ideal maps}.

More generally, considering all maps from $x\in X$ or various other particular maps (e.g. given by the Jacobian) should give information about the singularity $x\in X$ itself.

%%%%%%%%%%%%%%%%%%%%%%%%%%%%%%
\subsection{Potential applications and connections}

The study of singularities of rational maps naturally meets several other areas of algebraic
geometry and commutative algebra. We briefly mention some possible connections.

\emph{Cremona groups}. 
One obvious direction is birational transformation groups. For example, a birational map
$$
f\colon \mathbb P^n\dashrightarrow\mathbb P^n
$$
can be studied through the singularities of its graph and normalised graph. It would be
interesting to stratify Cremona transformations according to invariants such as
$\delta_x(f)$, $\lambda_x(f)$, or the singularities of $Y^\nu$, and to understand how these
invariants behave under composition. 

Can elements of certain type, e.g. linear type, generate the whole group? In dimension two, indeed linear type elements generate the whole group. 

\emph{Commutative algebra.}
There are also close connections with commutative algebra. If a rational map is given
locally by a base ideal $I$, then its graph and normalised graph are governed by the Rees
algebra and normalised Rees algebra of $I$. Thus questions about singularities of maps
interact naturally with integral closure, Rees valuations, syzygies, free resolutions and
homological properties of ideals. Conversely, geometric properties of the normalised graph
may give information about the corresponding ideals.

\emph{Enumerative geometry.}
In dimension three, ideals and their moduli also occur naturally in curve-counting theories.
For example, ideal sheaves of curves are basic objects in Donaldson--Thomas and stable
pair theories. Since an ideal determines a blowup and, after choosing generators, a rational
map, it is natural to ask whether invariants or moduli of map singularities have useful
interpretations in these theories. More generally, compactifications of spaces of rational
maps through their graphs may be related to Hilbert schemes, Quot schemes and moduli of
sheaves.

It is informative even to consider the surface map $f\colon \A^2\bir \PP^1$ given by $(x^m:y^n)$ in Example \ref{exa-map-(xm:yn)-A2} and extend it to dimension 3 via 
$$
(x^m,y^n)\subset k[x,y,z]
$$
or 
$$
(x^m,y^n,z)\subset k[x,y,z],
$$
work out its DT theory and compare with the invariants $\delta_o(f),\lambda_o(f),\theta_o(f)$. We leave it to the curious reader.

\emph{Combinatorial geometry.}
For toric and monomial maps, the theory becomes combinatorial. The normalised graph is
described by subdivisions of fans, while fibre degrees, discrepancies and thresholds can be
expressed using lattice points and piecewise linear functions. This suggests connections
with tropical geometry, Newton polyhedra and related problems in discrete and convex
geometry. It would be interesting to understand which invariants of map singularities have
natural tropical interpretations.

\emph{Dynamics.}
Another possible direction concerns birational dynamics. For a rational self-map
$$
f\colon X\dashrightarrow X,
$$
the indeterminacy loci of the iterates $f^m$ can become increasingly complicated.
Understanding the behaviour of the singularities of $f^m$ may give new quantitative
measures of this complexity and could interact with questions about algebraic stability and
degree growth.

\emph{Positive characteristic.}
Finally, many of the constructions in this paper have analogues in positive characteristic.
For instance, one may ask whether singularities of rational maps can be studied through
$F$-singularities of their graphs. In dimension $\le 3$ one can also use usual birational geometric singularity theory as we have done in this paper for surfaces.

%%%%%%%%%%%%%%%%%%%%%%%%%%%%%%%%%%%%%
%%%%%%%%%%%%%%%%%%%%%%%%%%%%%%%%%%%%%

\vspace{2cm}
%%%%%%%%%%%%%%%%%%%%%

\small
\textsc{Yau Mathematical Sciences Center, JingZhai, Tsinghua University, Hai Dian District, Beijing, China 100084  } \endgraf

\email{Email: birkar@tsinghua.edu.cn\\}

\end{document}